\documentclass[11pt,a4paper]{article}

\usepackage[pdftex,breaklinks,colorlinks,citecolor=blue,urlcolor=blue]{hyperref}
\usepackage{complexity}
\usepackage[utf8]{inputenc}
\usepackage{amsmath,amsthm}
\usepackage{amsfonts}
\usepackage{amssymb}
\usepackage{graphicx}
\usepackage{libertine}
\usepackage{fullpage}
\usepackage{verbatim}
\usepackage{xcolor}
\usepackage{enumitem}
\usepackage{todonotes}
\usepackage[T1]{fontenc}

\usepackage{tikz}
\usetikzlibrary{arrows,shapes,graphs}
\tikzstyle{vertex}=[circle,draw=black,fill=orange!40,minimum size=10pt,inner sep=0pt]
\tikzstyle{small vertex}=[circle,draw=black,fill=orange!40,minimum size=2pt,inner sep=2pt]
\tikzstyle{blue vertex}=[circle,draw=black,fill=blue!40,minimum size=2pt,inner sep=2pt]
\tikzstyle{odd vertex}=[draw=black,fill=orange!40,minimum size=10pt,inner sep=0pt]
\tikzstyle{even vertex} = [circle,draw=black,fill=red!40,minimum size=10pt,inner sep=0pt]
\tikzstyle{dual vertex}=[circle,draw=black,fill=yellow!25,minimum size=10pt,inner sep=0pt]
\tikzstyle{edge} = [draw,thick,-]
\tikzstyle{dashed edge} = [draw,thick,dashed,-]
\tikzstyle{oriented edge} = [draw,line width=1pt,->]
\tikzstyle{selected edge} = [draw,line width=3pt,->,red]
\tikzstyle{dual selected edge} = [draw,dashed,line width=3pt,->,green]
\tikzstyle{matched edge} = [draw,line width=4pt,-,blue]
\tikzstyle{other matched edge} = [draw,line width=4pt,-,red,dashed]
\tikzstyle{flip} = [draw,line width=2pt,<->]

\newtheorem{theorem}{Theorem}
\newtheorem{lemma}{Lemma}
\newtheorem{corollary}{Corollary}

\newtheorem{problem}{Problem}

\newtheorem{question}{Question}

\newcommand{\cI}{{\mathcal{I}}}

\newcommand{\cR}{{\mathcal{R}}}
\newcommand{\cO}{{\mathcal{O}}}
\newcommand{\cJ}{{\mathcal{J}}}
\providecommand{\cS}{{\mathcal{S}}} 
\renewcommand{\cS}{{\mathcal{S}}}   
\providecommand{\email}[1]{{\normalsize \tt {#1}}}  
\renewcommand{\email}[1]{{\normalsize \tt {#1}}}  

\DeclareMathOperator{\cov}{cov}
\DeclareMathOperator{\Int}{Int}
\DeclareMathOperator{\sgn}{sgn}

\graphicspath{{./graphics/}}

\definecolor{darkblue}{rgb}{0.1,0.1,0.7}
\let\emph\relax
\DeclareTextFontCommand{\emph}{\color{darkblue}}

\title{Flip distances between graph orientations}

\author{Oswin Aichholzer$^1$ \and Jean Cardinal$^2$ \and Tony Huynh$^2$\thanks{Supported by ERC Consolidator Grant 615640-ForEFront}  \and Kolja Knauer$^3$\thanks{Partially supported by ANR grants GATO: ANR-16-CE40-0009-01, DISTANCIA: ANR-17-CE40-0015, and CAPPS: ANR-17-CE40-0018} \and Torsten M\"utze$^4$ \and Raphael Steiner$^4$ \and Birgit Vogtenhuber$^1$\thanks{Partially supported by the Austrian Science Fund (FWF): I 3340-N35}}
\date{%
  $^1$TU Graz, Austria, \email{\{oaich,bvogt\}@ist.tugraz.at}\\%
  $^2$Universit\'e libre de Bruxelles (ULB), Belgium, \email{jcardin@ulb.ac.be, tony.bourbaki@gmail.com}\\%
  $^3$Universit\'e Aix-Marseille, France, \email{kolja.knauer@lis-lab.fr}\\%
  $^4$TU Berlin, Germany, \email{\{steiner,muetze\}@math.tu-berlin.de}\\
  \today
}

\begin{document}
\maketitle

\begin{abstract}
Flip graphs are a ubiquitous class of graphs, which encode relations induced on a set of combinatorial objects by elementary, local changes.
Skeletons of associahedra, for instance, are the graphs induced by quadrilateral flips in triangulations of a convex polygon.
For some definition of a flip graph, a natural computational problem to consider is the flip distance: Given two objects, what is
the minimum number of flips needed to transform one into the other?

We consider flip graphs on orientations of simple graphs, where flips consist of reversing the direction of some edges.
More precisely, we consider so-called $\alpha$-orientations of a graph $G$, in which every vertex $v$ has a specified outdegree $\alpha(v)$,
and a flip consists of reversing all edges of a directed cycle.
We prove that deciding whether the flip distance between two $\alpha$-orientations of a planar graph $G$ is at most two is \NP-complete.
This also holds in the special case of perfect matchings, where flips involve alternating cycles.
This problem amounts to finding geodesics on the common base polytope of two partition matroids, or, alternatively, on an alcoved polytope.
It therefore provides an interesting example of a flip distance question that is computationally intractable despite having a natural interpretation
as a geodesic on a nicely structured combinatorial polytope.

We also consider the dual question of the flip distance between graph orientations in which every cycle has a specified number of forward edges,
and a flip is the reversal of all edges in a minimal directed cut.
In general, the problem remains hard. However, if we restrict to flips that only change sinks into sources, or vice-versa, then
the problem can be solved in polynomial time. Here we exploit the fact that the flip graph is the cover graph of a distributive lattice.
This generalizes a recent result from Zhang, Qian, and Zhang (Acta. Math. Sin.-English Ser., 2019).
\end{abstract}

\section{Introduction}

The term \emph{flip} is commonly used in combinatorics to refer to an elementary, local, reversible operation that transforms one combinatorial object into another.
Such flip operations naturally yield a \emph{flip graph}, whose vertices are the considered combinatorial objects, and two of them are adjacent if they differ by a single flip.
A classical example is the flip graph of triangulations of a convex polygon~\cite{STT86,P14}; see Figure~\ref{fig:associahedron}.
The vertex set of this graph are all triangulations of the polygon, and two triangulations are adjacent if one can be obtained from the other by replacing the diagonal of a quadrilateral formed by two triangles by the other diagonal.
Similar flip graphs have also been investigated for triangulations of general point sets in the plane~\cite{DRS10}, triangulations of topological surfaces~\cite{N94}, and planar graphs~\cite{BH09,BV11}.
The \emph{flip distance} between two combinatorial objects is the minimum number of flips needed to transform one into the other.
It is known that computing the flip distance between two triangulations of a simple polygon~\cite{AMP15} or of a point set~\cite{LP15} is \NP-hard.
The latter is known to be fixed-parameter tractable~\cite{KSX17}.
On the other hand, the \NP-hardness of computing the flip distance between two triangulations of a convex polygon is a well-known open question~\cite{LZ98,R99,BP06,CJ09,CJ10,LEP10}.
Flip graphs involving other geometric configurations have also been studied, such as flip graphs of non-crossing perfect matchings of a point set in the plane, where flips are with respect to alternating 4-cycles~\cite{HHN02}, or alternating cycles of arbitrary length~\cite{HHNR05}.
Other flip graphs include the flip graph on plane spanning trees~\cite{AAHV07}, the flip graph of non-crossing partitions of a point set or dissections of a polygon~\cite{HHNO09}, the mutation graph of simple pseudoline arrangements~\cite{Rin-57}, the Eulerian tour graph of a Eulerian graph~\cite{ZG87}, and many others.
There is also a vast collection of interesting flip graphs for non-geometric objects, such as bitstrings, permutations, combinations, and partitions~\cite{FKMS18}.

\begin{figure}
\begin{center}
\includegraphics{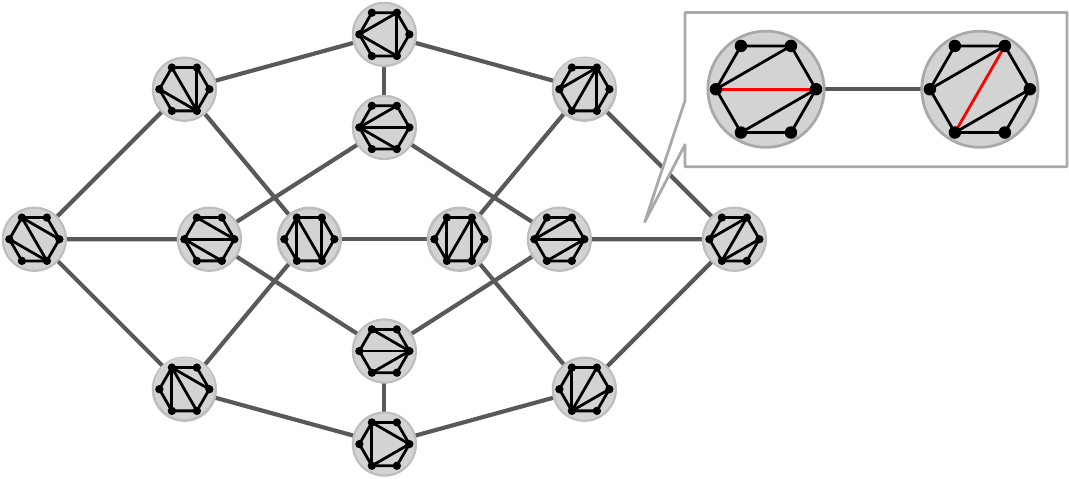}
\end{center}
\caption{The flip graph of triangulations of a convex polygon.}
\label{fig:associahedron}
\end{figure}

In essence, a flip graph provides the considered family of combinatorial objects with an underlying structure that reveals interesting properties about the objects.
It can also be a useful tool for proving that a property holds for all objects, by proving that one particularly nice object has the property, and that the property is preserved under flips.
Flip graphs are also an essential tool for solving fundamental algorithmic tasks such as random and exhaustive generation, see e.g.~\cite{AF96} and~\cite{PW98}.

The focus of the present paper is on flip graphs for orientations of graphs satisfying some constraints.
First, we consider so-called $\alpha$-orientations, in which the outdegree of every vertex is specified by a function $\alpha$, and the flip operation consists of reversing the orientation of all edges in a directed cycle.
We study the complexity of computing the flip distance between two such orientations.
An interesting special case of $\alpha$-orientations corresponds to perfect matchings in bipartite graphs, where flips involve alternating cycles.
We also consider the dual notion of $c$-orientations, in which the number of forward edges along each cycle is specified by a function $c$.
Here a flip consists of reversing all edges in a directed cut.
We also analyze the computational complexity of the flip distance problem in $c$-orientations.

There are several deep connections between flip graphs and polytopes.
Specifically, many interesting flip graphs arise as the ($1$-)skeleton of a polytope.
For instance, flip graphs of triangulations of a convex polygon are skeletons of associahedra~\cite{CSZ15}, and flip graphs of regular triangulations of a point set in the plane are skeletons of secondary polytopes (see~\cite[Chapter~5]{DRS10}).
Associahedra are generalized by \emph{quotientopes}~\cite{PS18}, whose skeletons yield flip graphs on rectangulations~\cite{CSS18}, bitstrings, permutations, and other combinatorial objects.
Moreover, flip graphs of acyclic orientations or strongly connected orientations of a graph are skeletons of graphical and co-graphical zonotopes, respectively~(see~\cite[Section~2]{P09}).
Similarly, as we show below, flip graphs on $\alpha$-orientations are skeletons of matroid intersection polytopes.
We also consider vertex flips in $c$-orientations, inducing flip graphs that are distributive lattices and in particular subgraphs of skeletons of certain distributive polytopes.
These polytopes specialize to flip polytopes of planar $\alpha$-orientations, are generalized by the polytope of tensions of a digraph, and form part of the family of \emph{alcoved polytopes}~(see~\cite{KF11}).

In the next section, we give the precise statements of the computational problems we consider, connections with previous work, and the statements of our main results.
In Section~\ref{sec:pm}, we give the proof of our first main result, showing that computing the flip distance between $\alpha$-orientations and between perfect matchings is \NP-hard even for planar graphs.
Section~\ref{sec:lattice} presents the proof of our second main result, where we give a polynomial time algorithm to compute the vertex flip distance between $c$-orientations. Finally, in Section~\ref{sec:jump} we show that computing the distance between $c$-orientations, when double vertex flips are also allowed, is \NP-hard.

\section{Problems and main results}

\paragraph{Flip distance between $\alpha$-orientations.}

Given a graph $G$ and some $\alpha:V(G) \rightarrow \mathbb{N}_0$, an \emph{$\alpha$-orientation} of $G$ is an orientation of the edges of $G$ in which every vertex $v$ has outdegree $\alpha (v)$.
An example for a graph and two $\alpha$-orientations for this graph is given in Figure~\ref{fig:flip}.
A flip of a directed cycle $C$ in some $\alpha$-orientation $X$ consists of the reversal of the orientation of all edges of $C$, as shown in the figure.
Edges with distinct orientations in two given $\alpha$-orientations $X$ and $Y$ induce a Eulerian subdigraph of both $X$ and $Y$.
They can therefore be partitioned into an edge-disjoint union of cycles in $G$ which are directed in both $X$ and $Y$.
Hence the reversal of each such cycle in $X$ gives rise to a flip sequence transforming $X$ into $Y$ and vice versa.
We may thus define the \emph{flip distance} between two $\alpha$-orientations $X$ and $Y$ to be the minimum number of cycles in a flip sequence transforming $X$ into $Y$.
We are interested in the computational complexity of determining the flip distance between two given $\alpha$-orientations.

\begin{figure}
\begin{center}
\begin{tikzpicture}[scale=.5, auto,swap]
    \begin{scope}[xshift=0cm]
    \foreach \pos/\name in {{(6,9)/a}, {(10,7)/d}, {(4,4.5)/e}, {(6,7)/f}, {(8,4.5)/k}}
        \node[vertex] (\name) at \pos {$1$};
    \foreach \pos/\name/\outdeg in {{(8,9)/b/1}, {(4,7)/c/2}, {(8,7)/g/2}, {(10,4.5)/h/1}, {(6,4.5)/j/2}}
        \node[vertex] (\name) at \pos {$\outdeg$};
    \foreach \source/\dest in {a/b,b/d,c/a,g/d,j/f,d/h,h/k}
        \path[oriented edge] (\source) -- (\dest);
    \foreach \source/\dest in {e/c,c/f,f/g,g/k,k/j,j/e}
        \path[selected edge] (\source) -- (\dest);
    \end{scope}
    \begin{scope}[xshift=11.5cm,yshift=5.5cm]
    \path[flip] (0,0) -- (5,0);
    \end{scope}
    \begin{scope}[xshift=14.5cm]
    \foreach \pos/\name in {{(6,9)/a}, {(10,7)/d}, {(4,4.5)/e}, {(6,7)/f}, {(8,4.5)/k}}
        \node[vertex] (\name) at \pos {$1$};
    \foreach \pos/\name/\outdeg in {{(8,9)/b/1}, {(4,7)/c/2}, {(8,7)/g/2}, {(10,4.5)/h/1}, {(6,4.5)/j/2}}
        \node[vertex] (\name) at \pos {$\outdeg$};
    \foreach \source/\dest in {a/b,b/d,c/a,g/d,j/f,d/h,h/k}
        \path[oriented edge] (\source) -- (\dest);
    \foreach \dest / \source in {e/c,c/f,f/g,g/k,k/j,j/e}
        \path[selected edge] (\source) -- (\dest);
    \end{scope}
\end{tikzpicture}
\end{center}
\caption{Two $\alpha$-orientations of a graph and a flip between them, where the values of $\alpha$ are depicted on the vertices.}
\label{fig:flip}
\end{figure}
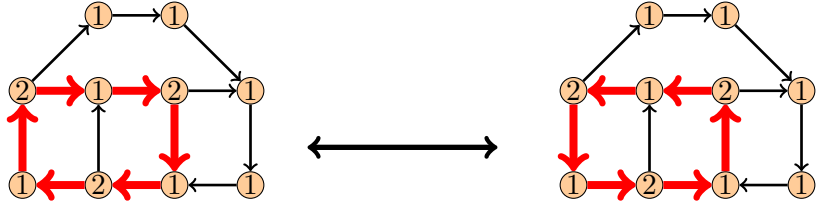

\begin{problem}
\label{pb:alphadist}
Given a graph $G$, some $\alpha:V(G) \rightarrow \mathbb{N}_0$, a pair $X,Y$ of $\alpha$-orientations of $G$ and an integer $k \ge 0$, decide whether the flip distance between $X$ and $Y$ is at most $k$.
\end{problem}

The crucial difficulty of this problem is that a shortest flip sequence transforming $X$ into $Y$ may flip edges that are oriented the same in $X$ and $Y$ an even number of times, to reach $Y$ with fewer flips compared to only flipping edges that are oriented differently in $X$ and $Y$; see the example in Figure~\ref{fig:XY}.
This motivates the following variant of the previous problem:

\begin{problem}
\label{pb:alphadistdiff}
Given $G,\alpha,X,Y,k$ as in Problem~\ref{pb:alphadist}, decide whether the flip distance between $X$ and $Y$ is at most $k$, where we only allow flipping edges that are oriented differently in $X$ and $Y$.
\end{problem}

\begin{figure}
\begin{center}
\begin{tikzpicture}[scale=.8, auto,swap]
    \begin{scope}[xshift=0cm]
    \foreach \pos/\name/\outdeg in {{(-1,0)/a/2}, {(1,0)/b/1}, {(2,1)/c/1}, {(2,-1)/d/1}, {(3,0)/e/2}, {(4,0)/f/1}, {(5,1)/g/1}, {(5,-1)/h/1}, {(6,0)/i/2}, {(7,0)/j/1}, {(8,1)/k/1}, {(8,-1)/l/1}, {(9,0)/m/2}, {(10,0)/n/1}, {(11,1)/o/1}, {(11,-1)/p/1}, {(12,0)/q/2}, {(14,0)/r/1}, {(6.5,2.5)/s/1}, {(6.5,-2.5)/t/1}}
        \node[vertex] (\name) at \pos {$\outdeg$};
    \foreach \source/\dest in {a/b,e/f,i/j,m/n,q/r}
        \path[oriented edge] (\source) -- (\dest);
    \foreach \dest / \source in {b/c,c/e,e/d,d/b,f/g,g/i,i/h,h/f,j/k,k/m,m/l,l/j,n/o,o/q,q/p,p/n}
        \path[selected edge] (\dest) -- (\source);
    \foreach \source/\dest in {r/s,s/a,a/t,t/r}
        \path[oriented edge] (\source) edge[bend right=18] node [left] {} (\dest);
    \node at (2,0) {$C_1$};
    \node at (5,0) {$C_2$};
    \node at (8,0) {$C_3$};
    \node at (11,0) {$C_4$};
    \node at (6.5,1.5) {$D_1$};
    \node at (6.5,-1.5) {$D_2$};
    \node at (1,2) {$D_3$};
    \end{scope}
\end{tikzpicture}
\end{center}
\caption{An $\alpha$-orientation $X$ of a graph.
The $\alpha$-orientation $Y$ obtained by flipping the four directed facial cycles $C_1,\ldots,C_4$ can be reached with fewer flips by flipping only the three directed facial cycles $D_1,D_2,D_3$ in this order.}
\label{fig:XY}
\end{figure}
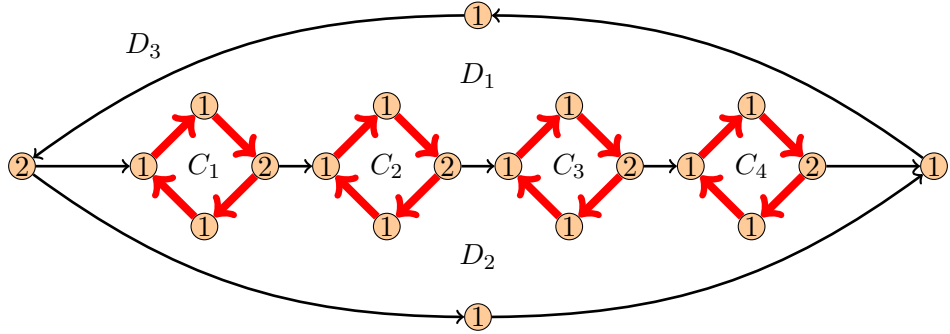

\paragraph{From $\alpha$-orientations to perfect matchings.}

The flexibility in choosing a function $\alpha$ for a set of $\alpha$-orientations on a graph allows us to capture numerous relevant combinatorial structures, some of which are listed below:
\begin{itemize}[leftmargin=4mm, noitemsep, topsep=3pt plus 3pt]
\item domino and lozenge tilings of a plane region \cite{Rem04,Thu90},
\item planar spanning trees \cite{GL86},
\item (planar) bipartite perfect matchings \cite{LZ03},
\item (planar) bipartite $d$-factors \cite{P02,F04},
\item Schnyder woods of a planar triangulation \cite{Bre00},
\item Eulerian orientations of a (planar) graph \cite{F04},
\item $k$-fractional orientations of a planar graph with specified outdegrees \cite{BF12},
\item contact representations of planar graphs with homothetic triangles, rectangles, and $k$-gons \cite{Fel13,GLP12,FSS18a,FSS18b}.
\end{itemize}

In the following, we focus on perfect matchings of bipartite graphs. Consider any bipartite graph $G$ with bipartition $(V_1,V_2)$ equipped with

$$\alpha:V(G) \rightarrow \mathbb{N}_0, \quad\alpha(x):=\begin{cases} 1 & \text{ if }  x \in V_1, \cr d_G(x)-1 & \text{ if } x \in V_2. \end{cases}$$

With this definition, in each $\alpha$-orientation of $G$, the edges directed from $V_1$ to $V_2$ form a perfect matching.
This is illustrated in Figure~\ref{fig:match}.
Conversely, given a perfect matching $M$ of $G$, orienting all edges of $M$ from $V_1$ to $V_2$ and all the other edges from $V_2$ to $V_1$ yields an $\alpha$-orientation of the above type.
Furthermore, the directed cycles in any $\alpha$-orientation of $G$ correspond to the alternating cycles in the associated perfect matching.
Flipping an alternating cycle in a perfect matching corresponds to exchanging matching-edges and non-matching-edges.
An example of the flip graph of perfect matchings of a graph is given in Figure~\ref{fig:flipgraph}.
In this special case, Problem~\ref{pb:alphadist} boils down to:

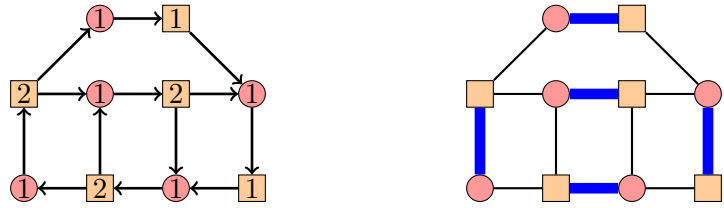
\begin{figure}
\begin{center}
\begin{tikzpicture}[scale=.5, auto,swap]
    \begin{scope}[xshift=0cm]
    \foreach \pos/\name in {{(6,9)/a}, {(10,7)/d}, {(4,4.5)/e}, {(6,7)/f}, {(8,4.5)/k}}
        \node[even vertex] (\name) at \pos {$1$};
    \foreach \pos/\name/\outdeg in {{(8,9)/b/1}, {(4,7)/c/2}, {(8,7)/g/2}, {(10,4.5)/h/1}, {(6,4.5)/j/2}}
        \node[odd vertex] (\name) at \pos {$\outdeg$};
    \foreach \source/\dest in {a/b,b/d,c/a,c/f,f/g,g/d,j/f,k/j,g/k,e/c,d/h,j/e,h/k}
        \path[oriented edge] (\source) -- (\dest);
    \end{scope}
    \begin{scope}[xshift=12cm]
    \foreach \pos/\name in {{(6,9)/a}, {(10,7)/d}, {(4,4.5)/e}, {(6,7)/f}, {(8,4.5)/k}}
        \node[even vertex] (\name) at \pos {};
    \foreach \pos/\name/\outdeg in {{(8,9)/b/1}, {(4,7)/c/2}, {(8,7)/g/2}, {(10,4.5)/h/1}, {(6,4.5)/j/2}}
        \node[odd vertex] (\name) at \pos {};
    \foreach \source/\dest in {a/b,b/d,a/c,c/f,f/g,g/d,f/j,j/k,k/g,c/e,h/d,j/e,h/k}
        \path[edge] (\source) -- (\dest);
    \foreach \source/\dest in {a/b,f/g,k/j,e/c,d/h}
        \path[matched edge] (\source) -- (\dest);
    \end{scope}
\end{tikzpicture}
\end{center}
\caption{An $\alpha$-orientation of a bipartite graph and the corresponding perfect matching.}
\label{fig:match}
\end{figure}

\begin{problem}
\label{pb:bipmatchingdist}
Given a bipartite graph $G$, a pair $X,Y$ of perfect matchings in $G$ and an integer $k \ge 0$, decide whether the flip distance between $X$ and $Y$ is at most $k$.
\end{problem}

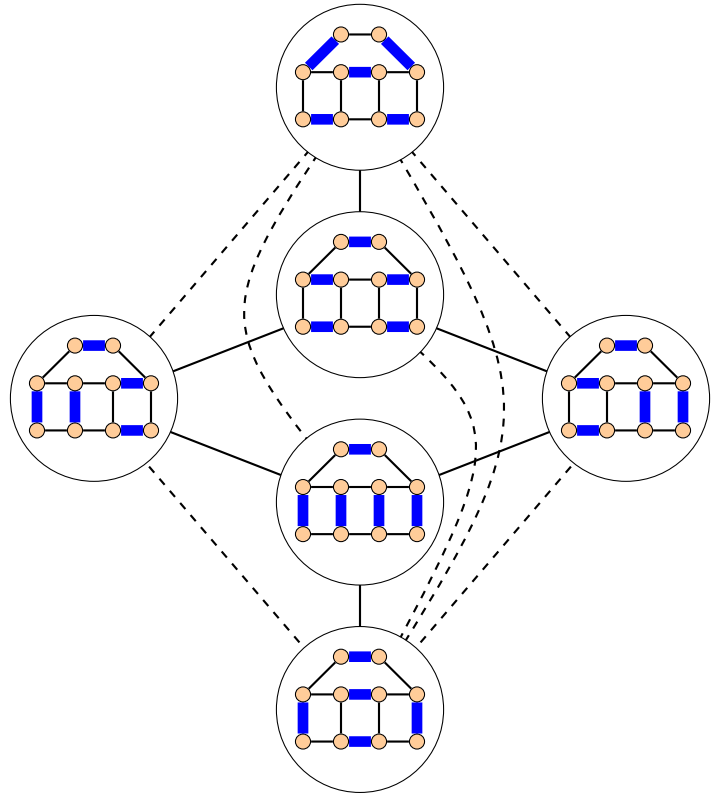
\begin{figure}
\begin{center}
\newcommand{\basegraph}
{
    \draw[fill=white] (7,6.2) circle (4.4cm);
    \foreach \pos/\name in {{(6,9)/a}, {(10,7)/d}, {(4,4.5)/e}, {(6,7)/f}, {(8,4.5)/k}}
        \node[small vertex] (\name) at \pos {};
    \foreach \pos/\name/\outdeg in {{(8,9)/b/1}, {(4,7)/c/2}, {(8,7)/g/2}, {(10,4.5)/h/1}, {(6,4.5)/j/2}}
        \node[small vertex] (\name) at \pos {};
    \foreach \source/\dest in {a/b,b/d,a/c,c/f,f/g,g/d,f/j,j/k,k/g,c/e,h/d,k/h,j/e}
        \path[edge] (\source) -- (\dest);
}
\begin{tikzpicture}[scale=.25, auto,swap]
    \foreach \pos/\name in {{(12,2)/a}, {(12,13)/b}, {(12,24)/c}, {(12,35)/d}, {(-2,18.5)/e}, {(26,18.5)/f}}
        \node (\name) at \pos {};
    \foreach \source/\dest in {a/b,c/d,b/e,b/f,c/e,c/f}
        \path[edge] (\source) -- (\dest);
    \foreach \source/\dest in {a/e,a/f,d/e,d/f}
        \path[dashed edge] (\source) -- (\dest);
    \path[dashed edge] (d) .. controls (22,18) .. (a);
    \path[dashed edge] (d) .. controls (4,23) .. (b);
    \path[dashed edge] (c) .. controls (20,17) .. (a);
    \begin{scope}[xshift=5cm,yshift=-4cm]
    \basegraph
        \foreach \source/\dest in {a/b,f/g,c/e,d/h,j/k}
        \path[matched edge] (\source) -- (\dest);
    \end{scope}
    \begin{scope}[xshift=5cm,yshift=7cm]
    \basegraph
        \foreach \source/\dest in {a/b,d/h,e/c,f/j,g/k}
        \path[matched edge] (\source) -- (\dest);
    \end{scope}
    \begin{scope}[xshift=5cm,yshift=18cm]
    \basegraph
        \foreach \source/\dest in {c/f,g/d,k/h,a/b,j/e}
        \path[matched edge] (\source) -- (\dest);
    \end{scope}
    \begin{scope}[xshift=5cm,yshift=29cm]
    \basegraph
        \foreach \source/\dest in {a/c,b/d,f/g,j/e,h/k}
        \path[matched edge] (\source) -- (\dest);
    \end{scope}
    \begin{scope}[xshift=-9cm,yshift=12.5cm]
    \basegraph
        \foreach \source/\dest in {a/b,e/c,f/j,g/d,k/h}
        \path[matched edge] (\source) -- (\dest);
    \end{scope}
    \begin{scope}[xshift=19cm,yshift=12.5cm]
    \basegraph
        \foreach \source/\dest in {a/b,c/f,j/e,g/k,d/h}
        \path[matched edge] (\source) -- (\dest);
    \end{scope}
\end{tikzpicture}
\end{center}
\caption{The flip graph of perfect matchings of a graph. The solid edges indicate flips along facial cycles, and the dashed edges indicate flips along non-facial cycles.}
\label{fig:flipgraph}
\end{figure}

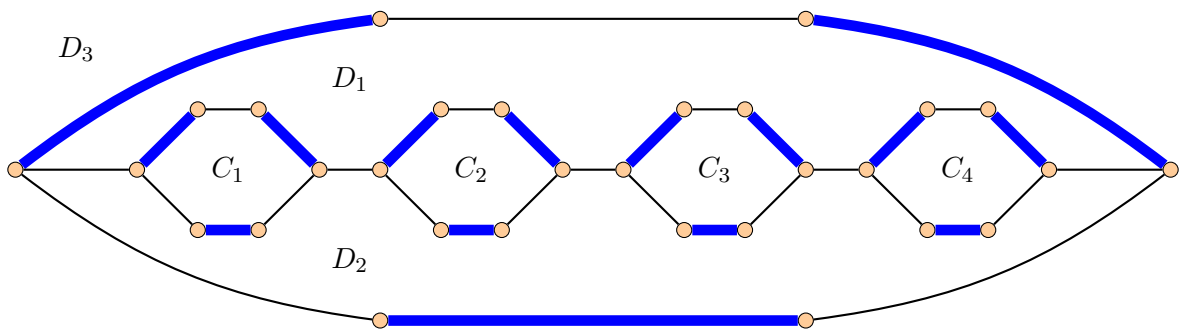
\begin{figure}
\begin{center}
\begin{tikzpicture}[scale=.8, auto,swap]
    \begin{scope}[xshift=0cm]
    \foreach \pos/\name in {{(0,0)/a}, {(19,0)/b}, {(6,2.5)/c}, {(13,2.5)/e}, {(6,-2.5)/d}, {(13,-2.5)/f}, {(2,0)/a1}, {(3,1)/a2}, {(3,-1)/a6}, {(4,1)/a3}, {(4,-1)/a5}, {(5,0)/a4}, {(6,0)/b1}, {(7,1)/b2}, {(7,-1)/b6}, {(8,1)/b3}, {(8,-1)/b5}, {(9,0)/b4}, {(10,0)/c1}, {(11,1)/c2}, {(11,-1)/c6}, {(12,1)/c3}, {(12,-1)/c5}, {(13,0)/c4}, {(14,0)/d1}, {(15,1)/d2}, {(15,-1)/d6}, {(16,1)/d3}, {(16,-1)/d5}, {(17,0)/d4}}
        \node[small vertex] (\name) at \pos {};
    \foreach \source/\dest in {a1/a2,a3/a4,a5/a6,b1/b2,b3/b4,b5/b6,c1/c2,c3/c4,c5/c6,d1/d2,d3/d4,d5/d6,d/f}
        \path[matched edge] (\source) -- (\dest);
    \foreach \source/\dest in {c/a,b/e}
        \path[matched edge] (\source) edge[bend right=15] node [left] {} (\dest);
    \foreach \source/\dest in {a2/a3,a4/a5,a1/a6,b2/b3,b4/b5,b1/b6,c2/c3,c4/c5,c1/c6,d2/d3,d4/d5,d1/d6,c/e,a/a1,a4/b1,b4/c1,c4/d1,b/d4}
        \path[edge] (\source) -- (\dest);
    \foreach \source/\dest in {a/d,f/b}
        \path[edge] (\source) edge[bend right=15] node [left] {} (\dest);
    \node at (3.5,0) {$C_1$};
    \node at (7.5,0) {$C_2$};
    \node at (11.5,0) {$C_3$};
    \node at (15.5,0) {$C_4$};
    \node at (5.5,1.5) {$D_1$};
    \node at (5.5,-1.5) {$D_2$};
    \node at (1,2) {$D_3$};
    \end{scope}
\end{tikzpicture}
\end{center}
\caption{A perfect matching $X$ in a graph.
The perfect matching $Y$ obtained by flipping the four alternating facial cycles $C_1,\ldots,C_4$ can be reached with fewer flips by flipping only the three alternating facial cycles $D_1,D_2,D_3$ in this order.}
\label{fig:M1M2}
\end{figure}

The example from Figure~\ref{fig:XY} can be easily modified to show that when transforming $X$ into $Y$ using the fewest number of flips, we may have to flip alternating cycles that are not in the symmetric difference of $X$ and $Y$; see the example in Figure~\ref{fig:M1M2}.
If we restrict the flips to only use cycles in the symmetric difference of $X$ and $Y$, then the problem of finding the flip distance becomes trivial, as the symmetric difference is a collection of disjoint cycles, and each of them has to be flipped, so Problem~\ref{pb:alphadistdiff} is trivial for perfect matchings.

\paragraph{Flip graphs and matroid intersection polytopes.}

We proceed to give a geometric interpretation of the flip distance between $\alpha$-orientations as the distance in the skeleton of a 0/1-polytope.

Recall that a \emph{matroid} is an abstract simplicial complex $(E,\cI)$, where $\cI\subseteq 2^E$ satisfies the independent set augmentation property.
The elements of $\cI$ are called \emph{independent sets}.
A \emph{base} of the matroid is an inclusionwise maximal independent set.

It is well-known that perfect matchings in a bipartite graph $G=(V_1\cup V_2,E)$ are common bases of two partition matroids $(E,\cI_1)$ and $(E,\cI_2)$, in which a set of edges is independent if no two share an endpoint in $V_1$, or, respectively, in $V_2$.

Similarly, $\alpha$-orientations can be defined as common bases of two partition matroids.
In this case, every edge of the graph $G$ is replaced by a pair of parallel arcs, one for each possible orientation of the edge.
One matroid encodes the constraint that for every edge exactly one orientation is chosen.
The second matroid encodes the constraint that each vertex $v$ has exactly $\alpha (v)$ outgoing arcs.

The \emph{common base polytope} of two matroids is a 0/1-polytope obtained as the convex hull of the characteristic vectors of the common bases.
Adjacency of two vertices of this polytope has been characterized by Frank and Tardos~\cite{FT88}.
A shorter proof was given by Iwata~\cite{I02}.
We briefly recall their result in the next theorem.
To state the theorem, consider a matroid $M=(E,\cI)$, a base $B\in\cI$, and a subset $F\subseteq E$.
The \emph{exchangeability graph} $G(B,F)$ of $M$ is a bipartite graph with $B\setminus F$ and $F\setminus B$ as vertex bipartition, and edge set $\{ ij \mid B\setminus\{i\}\cup\{j\} \text{ is a basis}\}$.
This definition and the theorem are illustrated in Figure~\ref{fig:exchange} for the two partition matroids whose common bases are perfect matchings of a graph.

\begin{theorem}[\cite{FT88,I02}]
\label{thm:polytope}
For two matroids $M^+ = (E,\cI^+)$ and $M^- = (E,\cI^-)$, two common bases $A,\!B\in \cI^+\cap\cI^-$ are adjacent on the common base polytope if and only if all the following conditions hold:
\begin{enumerate}[label=(\roman*),leftmargin=8mm, noitemsep, topsep=3pt plus 3pt]
\item the exchangeability graph $G(A,B)$ of $M^+$ has a unique perfect matching $P^+$,
\item the exchangeability graph $G(B,A)$ of $M^-$ has a unique perfect matching $P^-$,
\item $P^+ \cup P^-$ is a single cycle.
\end{enumerate}
\end{theorem}

From this theorem we can conclude that the flip graphs we consider on perfect matchings and $\alpha$-orientations are precisely the skeletons of the corresponding polytopes of common bases.

It is interesting to compare Problems~\ref{pb:alphadist} and~\ref{pb:bipmatchingdist} with the analogous problems for other families of matroid polytopes.
For instance, it is known that for two bases $A,B$ of a matroid, the exchangeability graph $G(A,B)$ has a perfect matching~\cite{B69}.
Hence $A$ can be transformed into $B$ by performing $|A\Delta B|/2$ exchanges of elements (where $A\Delta B$ is the symmetric difference of $A$ and $B$), which is also the distance in the skeleton of the base polytope of the matroid.
On the other hand, the problem of computing the flip distance between two triangulations of a convex polygon amounts to computing distances in skeletons of associahedra, which are known to be polymatroids (see~\cite{AA17} and references therein).
This problem is neither known to be in \P\ nor known to be \NP-hard.
Also note that for other families of combinatorial polytopes, testing adjacency is already intractable.
This is the case for instance for the polytope of the Traveling Salesman Problem (TSP)~\cite{P78}, whose skeleton is known to have diameter at most 4~\cite{RC98}.
On the other hand, the corresponding polytope is known to be the common base polytope of three matroids.

Another important class of combinatorial polytopes are \emph{alcoved polytopes}, see~\cite{LP07}. It is known that the flip graphs of planar $\alpha$-orientations are skeletons of alcoved polytopes, see~\cite{KF11}. Thus, by our results below, flip distances in this class are also \NP-hard to compute.

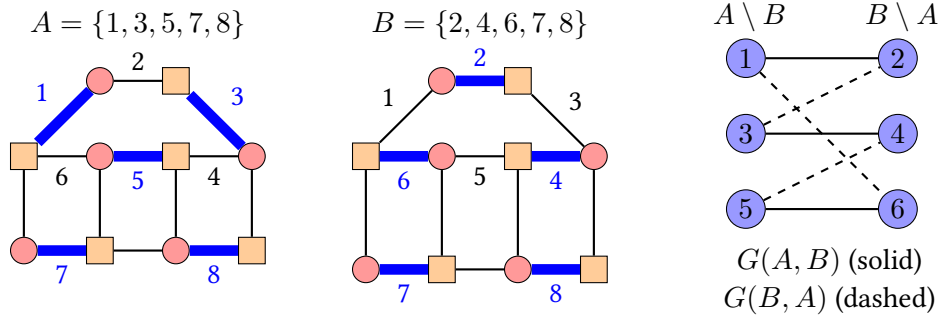
\begin{figure}
\begin{center}
\begin{tikzpicture}[scale=.5, auto,swap]
    \begin{scope}
    \node at (7,10.5) {$A=\{1,3,5,7,8\}$};
    \foreach \pos/\name in {{(6,9)/a}, {(10,7)/d}, {(4,4.5)/e}, {(6,7)/f}, {(8,4.5)/k}}
        \node[even vertex] (\name) at \pos {};
    \foreach \pos/\name/\outdeg in {{(8,9)/b/1}, {(4,7)/c/2}, {(8,7)/g/2}, {(10,4.5)/h/1}, {(6,4.5)/j/2}}
        \node[odd vertex] (\name) at \pos {};
    \foreach \source/\dest in {c/e,f/j,g/k,d/h,j/k}
        \path[edge] (\source) -- (\dest);
    \foreach \source/\dest/\lbl in {b/a/2,g/d/4,c/f/6}
        \path[edge] (\source) edge node {\lbl} (\dest);
    \foreach \source/\dest/\lbl in {a/c/1,d/b/3,f/g/5,e/j/7,k/h/8}
        \path[matched edge] (\source) edge node {\lbl} (\dest);
    \end{scope}
    \begin{scope}[xshift=9cm]
    \node at (7,10.5) {$B=\{2,4,6,7,8\}$};
    \foreach \pos/\name in {{(6,9)/a}, {(10,7)/d}, {(4,4)/e}, {(6,7)/f}, {(8,4)/k}}
        \node[even vertex] (\name) at \pos {};
    \foreach \pos/\name/\outdeg in {{(8,9)/b/1}, {(4,7)/c/2}, {(8,7)/g/2}, {(10,4)/h/1}, {(6,4)/j/2}}
        \node[odd vertex] (\name) at \pos {};
    \foreach \source/\dest in {c/e,f/j,g/k,d/h,j/k}
        \path[edge] (\source) -- (\dest);
    \foreach \source/\dest/\lbl in {a/c/1,d/b/3,f/g/5}
        \path[edge] (\source) edge node {\lbl} (\dest);
    \foreach \source/\dest/\lbl in {e/j/7,k/h/8,c/f/6,g/d/4,b/a/2}
        \path[matched edge] (\source) edge node {\lbl} (\dest);
    \end{scope}
    \begin{scope}[xshift=23cm,yshift=5.6cm]
    \foreach \pos/\name in {{(0,0)/5}, {(0,2)/3}, {(0,4)/1}, {(4,0)/6}, {(4,2)/4}, {(4,4)/2}}
	\node[blue vertex] (\name) at \pos {$\name$};
    \foreach \source/\dest in {1/2,3/4,5/6}
         \path[edge] (\source) -- (\dest);
    \foreach \source/\dest in {1/6,3/2,5/4}
         \path[dashed edge] (\source) -- (\dest);
    \node at (0.1,5.1) {$A\setminus B$};
    \node at (4.1,5.1) {$B\setminus A$};
    \node at (2.2,-1.4) {$G(A,B)$ (solid)};
    \node at (2.2,-2.4) {$G(B,A)$ (dashed)};
    \end{scope}
\end{tikzpicture}
\end{center}
\caption{Two common bases $A$ and $B$ (left and middle) of the matroids $M^+$ and $M^-$, where $M^+$ and $M^-$ have as independent sets all subsets of edges of the graph where no two share an endpoint in the set of circled vertices, or the set of squared vertices, respectively.
The right hand side shows the exchangeability graphs $G(A,B)$ of $M^+$ (solid edges) and $G(B,A)$ of $M^-$ (dashed edges).
As the conditions of Theorem~\ref{thm:polytope} are met, the two bases are adjacent in the common base polytope, and adjacent in the flip graph shown in Figure~\ref{fig:flipgraph}.}
\label{fig:exchange}
\end{figure}

\paragraph{Hardness of flip distance between perfect matchings and $\alpha$-orientations.}

We prove that Problem~\ref{pb:bipmatchingdist} is \NP-complete, even for 2-connected bipartite subcubic planar graphs and $k=2$.
This clearly implies that Problem~\ref{pb:alphadist} is \NP-complete as well.

\begin{theorem}
\label{thm:hardness}
Given a 2-connected bipartite subcubic planar graph $G$ and a pair $X,Y$ of perfect matchings in $G$, deciding whether the flip distance between $X$ and $Y$ is at most two is \NP-complete.
\end{theorem}

As direct consequences of the proof of Theorem~\ref{thm:hardness} we get:

\begin{corollary}
Unless $\P=\NP$, deciding whether the flip distance between two perfect matchings is at most $k$ is not fixed-parameter tractable with respect to parameter $k$.
\end{corollary}

\begin{corollary}
Unless $\P=\NP$, the flip distance between two perfect matchings is not approximable within a multiplicative factor $3/2-\epsilon$ in polynomial time, for any $\epsilon >0$.
\end{corollary}

We also prove that Problem~\ref{pb:alphadistdiff} is \NP-complete, even for $4$-regular graphs and $k=2$.

\begin{theorem}
\label{thm:hardnessdiff}
Given a $4$-regular graph $G$ and a pair $X,Y$ of $\alpha$-orientations of $G$, deciding whether the flip distance between $X$ and $Y$ is at most two is \NP-complete.  Moreover, the problem remains \NP-complete if we only allow flipping edges that are oriented differently in $X$ and $Y$
\end{theorem}

The proofs of Theorem~\ref{thm:hardness} and \ref{thm:hardnessdiff} are presented in Section~\ref{sec:pm}.

\paragraph{From $\alpha$-orientations in planar graphs to $c$-orientations.}

In what follows, we generalize the problem, via planar duality, to flip distances in so-called $c$-orientations.

Consider an arbitrary 2-connected plane graph $G$ and its planar dual $G^\ast$.
Then for any orientation $D$ of the edges of $G$, the directed dual $D^\ast$ of $D$ is obtained by orienting any dual edge forward if it crosses a left-to-right arc in $D$ in a simultaneous plane embedding of $G$ and $G^\ast$, and backward otherwise; see Figure~\ref{fig:duality}.
Edge sets of directed cycles in $D$ correspond to edge sets of minimal \emph{directed cuts} in $D^\ast$ and vice-versa.
Hence $D$ is acyclic (respectively, strongly connected) if and only if $D^\ast$ is strongly connected (respectively, acyclic).
A \emph{directed vertex cut} is a cut consisting of all edges incident to a sink or a source vertex.
Directed facial cycles in $D$ are in bijection with the directed vertex cuts in $D^\ast$, and vice versa.
The unbounded face in the plane embedding of $D$ can be chosen such that it corresponds to a fixed vertex $\top$ in $D^\ast$.

Let $D$ be an $\alpha$-orientation of $G$.
Given a minimal cut in $D$ separating $U \subseteq V(D)$ from $\overline{U}:=V(D) \setminus U$, we denote by $\delta^+(U)$ the edges pointing from $U$ to $\overline{U}$ in $D$.
We also let $d_D^+(v)$ denote the outdegree of vertex $v$ in $D$.
We have
$$|\delta^+(U)|=\sum_{v \in U}{d_D^+(v)}-|E(G[U])|=\sum_{v \in U}{\alpha(v)}-|E(G[U])|,$$
which only depends on $\alpha$ and $G$.
Consequently, the set of orientations of $G^\ast$ which are directed duals of $\alpha$-orientations of $G$ can be characterized by the property that for every cycle $C$ in $G^\ast$, the number of edges in clockwise direction is fixed by a certain value $c(C)$ independent of the orientation.
The flip operation between $\alpha$-orientations of $D$ consists of the reversal of a directed cycle.
In the corresponding set of dual orientations of $D^\ast$, this translates to the reversal of the orientations of the edges in a minimal directed cut, as shown on Figure~\ref{fig:duality}.

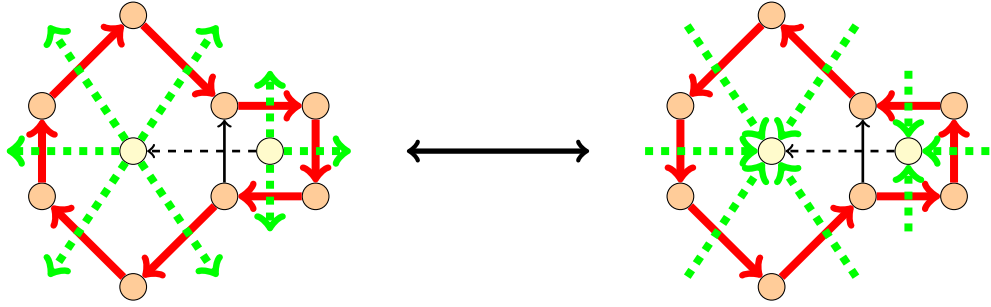
\begin{figure}
\begin{center}
\begin{tikzpicture}[scale=.6, auto,swap]
    \begin{scope}[xshift=0cm]
    \foreach \pos/\name in {{(2,6)/e}, {(6,6)/f}, {(8,4)/k}, {(4,2)/m},{(4,8)/c}, {(8,6)/g}, {(2,4)/i}, {(6,4)/j}}
        \node[vertex] (\name) at \pos {};
    \foreach \pos/\name in {{(4,5)/d1}, {(7,5)/d2}}
        \node[dual vertex] (\name) at \pos {};
    \foreach \pos/\name in {{(1,5)/s}, {(2,8)/u}, {(6,8)/v}, {(7,7)/w}, {(9,5)/x}, {(2,2)/y}, {(6,2)/z}, {(7,3)/t}}
        \node (\name) at \pos {};
    \path[oriented edge] (j) -- (f);
    \path[oriented edge, dashed] (d2) -- (d1);
    \foreach \source/\dest in {i/e,m/i,j/m,k/j,g/k,f/g,c/f,e/c}
        \path[selected edge] (\source) -- (\dest);
    \foreach \source/\dest in {d1/s,d1/u,d1/v,d1/y,d1/z,d2/w,d2/x,d2/t}
        \path[dual selected edge] (\source) -- (\dest);
    \end{scope}
    \begin{scope}[xshift=10cm,yshift=5cm]
    \path[flip] (0,0) -- (4,0);
    \end{scope}
    \begin{scope}[xshift=14cm]
    \foreach \pos/\name in {{(2,6)/e}, {(6,6)/f}, {(8,4)/k}, {(4,2)/m},{(4,8)/c}, {(8,6)/g}, {(2,4)/i}, {(6,4)/j}}
        \node[vertex] (\name) at \pos {};
    \foreach \pos/\name in {{(4,5)/d1}, {(7,5)/d2}}
        \node[dual vertex] (\name) at \pos {};
    \foreach \pos/\name in {{(1,5)/s}, {(2,8)/u}, {(6,8)/v}, {(7,7)/w}, {(9,5)/x}, {(2,2)/y}, {(6,2)/z}, {(7,3)/t}}
        \node (\name) at \pos {};
    \path[oriented edge] (j) -- (f);
    \path[oriented edge, dashed] (d2) -- (d1);
    \foreach \source/\dest in {c/e,f/c,g/f,k/g,j/k,m/j,i/m,e/i}
        \path[selected edge] (\source) -- (\dest);
    \foreach \source/\dest in {s/d1,u/d1,v/d1,y/d1,z/d1,w/d2,x/d2,t/d2}
        \path[dual selected edge] (\source) -- (\dest);
    \end{scope}
\end{tikzpicture}
\end{center}
\caption{Duality between flips in $\alpha$-orientations (solid edges) and in $c$-orientations (dashed edges).}
\label{fig:duality}
\end{figure}

The same notion has been investigated more generally without planarity conditions under the name of $c$-orientations by Propp~\cite{P02} and Knauer~\cite{K07}.
Given a graph $G$, we can fix an arbitrary direction of traversal for each cycle $C$.
Given a graph and an assignment $c(C) \in \mathbb{N}_0$ to each cycle in $G$, one may define a \emph{$c$-orientation} of $G$ to be an orientation having exactly $c(C)$ edges in forward direction for every cycle $C$ in $G$.
Note that it is sufficient to define the function $c$ on a cycle basis of $G$, which consists of linearly many cycles.
The \emph{flip} operation on the set $\cR_c$ of such $c$-orientations of a graph is defined as the reversal of all edges in a minimal directed cut.
It is not difficult to see that flips make the set of $c$-orientations of a graph connected (this will be noted in Section~\ref{sec:lattice}).

From the duality between planar $\alpha$-orientations and planar $c$-orientations, determining flip distances between $\alpha$-orientations of 2-connected planar graphs reduces to determining flip distances between the dual $c$-orientations. Note that planar duals of bipartite graphs are exactly the Eulerian planar graphs. Theorem~\ref{thm:hardness} therefore directly yields:

\begin{corollary}
Given a Eulerian planar multigraph $G$ and a pair $X,Y$ of $c$-orientations of $G$, deciding whether the flip distance between $X$ and $Y$ is at most two is \NP-complete.
\end{corollary}

\paragraph{$c$-orientations and distributive lattices.}

A more local operation consists of flipping only directed vertex cuts, induced by sources and sinks, excluding a fixed vertex $\top$.
We will refer to this special case as a \emph{vertex flip}.
Specifically, given a pair of $c$-orientations $X,Y$ of a graph $G$ with a fixed vertex $\top$, we aim to transform $X$ into $Y$ using only vertex flips at vertices distinct from $\top$.

A $c$-orientation $X$ of $G$ may not be acyclic, that is, there is a cycle $C$ in $G$ which is directed in $X$.
According to the definition of a $c$-orientation, this means that $C$ keeps the same orientation in every $c$-orientation of $G$.
Consequently, any (minimal) directed cut in a $c$-orientation of $G$ is disjoint from $E(C)$.
Contracting the cycle $C$ in $G$, we end up with a smaller graph $G'$ containing the same (minimal) directed cuts, such that the $c$-orientations of $G$ are determined by their corresponding orientations on $G'$.
We can therefore safely assume that the $c$-orientations that we consider are all acyclic.
Similarly, $G$ will be assumed to be connected.

\begin{problem}
\label{pb:vertexflip}
Given a connected graph $G$ with a fixed vertex $\top$ and a pair $X,Y$ of acyclic $c$-orientations, what is the length of a shortest vertex flip sequence transforming $X$ into $Y$?
\end{problem}

We should convince ourselves that under the assumptions made above, every pair of $c$-orientations is reachable from each other by vertex flips.
This property is provided in a much stronger way by a distributive lattice structure on the set $\cR_c$; see Figure~\ref{fig:lattice}.
The next theorem is a special case of Theorem~1 in Propp~\cite{P02} where the $c$-orientations are acyclic.

\begin{theorem}[\cite{P02,K07}]
\label{thm:DL}
Let $G$ be a graph with fixed vertex $\top$ and $\cR_c$ a set of acyclic $c$-orientations of $G$.
Then the partial order $\le_c$ on $\cR_c$ in which $Y$ covers $X$ if and only if $Y$ can be obtained from $X$ by flipping a source defines a distributive lattice on $\cR_c$.
\end{theorem}

Hence Problem~\ref{pb:vertexflip} consists of finding shortest paths in the cover graph of a distributive lattice, where the size of the lattice can be exponential in the size of the input $G$.

\begin{figure}
\newcommand{\basegraph}[4]
{
    \draw[fill=white] (0,1) ellipse (5.4cm and 3.9cm);
    \foreach \pos/\name/\lbl in {{(0,2)/b/#1}, {(-2,0)/d/#2}, {(0,0)/e/#3}, {(2,0)/f/#4}}
        \node[dual vertex] (\name) at \pos {\lbl};
    \node[dual vertex] (a) at (0,4) {$\top$};
    \foreach \pos/\name in {{(1,-3)/c1}, {(4.3,-1.5)/c2}, {(4.3,0)/c3}, {(4.3,3.5)/c4}, {(-4,0)/c5}, {(-4,2)/c6}, {(4,0)/c7}, {(4,2)/c8}}
        \node[shape=coordinate] (\name) at \pos {};
}
\begin{center}
\begin{tikzpicture}[scale=.3]
\begin{scope}[xshift=5cm, yshift=1cm]
    \path[edge] (0,-2) -- (0,7);
    \path[edge] (0,7) -- (-10,14);
    \path[edge] (0,7) -- (10,14);
    \path[edge] (-10,14) -- (0,21);
    \path[edge] (10,14) -- (0,21);
    \path[edge] (0,21) -- (0,30);
\end{scope}
\begin{scope}[xshift=5cm, yshift=-2cm]
\basegraph{0}{0}{0}{0}
    \foreach \source/\dest in {b/a,d/b,e/b,f/b,e/d,e/f}
        \path[oriented edge] (\source) -- (\dest);
    \path[oriented edge] (e) .. controls (c1) and (c2) .. (c3) .. controls (c4) .. (a);        
    \path[oriented edge] (d) .. controls (c5) and (c6) .. (a);
    \path[oriented edge] (f) .. controls (c7) and (c8) .. (a);
\end{scope}
\begin{scope}[xshift=5cm, yshift=7cm]
\basegraph{0}{0}{1}{0}
    \foreach \source/\dest in {b/a,d/b,b/e,f/b,d/e,f/e}
        \path[oriented edge] (\source) -- (\dest);
    \path[oriented edge] (a) .. controls (c4) .. (c3) .. controls (c2) and (c1) .. (e);        
    \path[oriented edge] (d) .. controls (c5) and (c6) .. (a);
    \path[oriented edge] (f) .. controls (c7) and (c8) .. (a);
\end{scope}
\begin{scope}[xshift=5cm, yshift=21cm]
\basegraph{0}{1}{1}{1}
    \foreach \source/\dest in {b/a,b/d,b/e,b/f,e/d,e/f}
        \path[oriented edge] (\source) -- (\dest);
    \path[oriented edge] (a) .. controls (c4) .. (c3) .. controls (c2) and (c1) .. (e);        
    \path[oriented edge] (a) .. controls (c6) and (c5) .. (d);
    \path[oriented edge] (a) .. controls (c8) and (c7) .. (f);
\end{scope}
\begin{scope}[xshift=5cm, yshift=30cm]
\basegraph{1}{1}{1}{1}
    \foreach \source/\dest in {a/b,d/b,e/b,f/b,e/d,e/f}
        \path[oriented edge] (\source) -- (\dest);
    \path[oriented edge] (a) .. controls (c4) .. (c3) .. controls (c2) and (c1) .. (e);        
    \path[oriented edge] (a) .. controls (c6) and (c5) .. (d);
    \path[oriented edge] (a) .. controls (c8) and (c7) .. (f);
\end{scope}
\begin{scope}[xshift=-5cm, yshift=14cm]
\basegraph{0}{0}{1}{1}
    \foreach \source/\dest in {b/a,d/b,b/e,b/f,d/e,e/f}
        \path[oriented edge] (\source) -- (\dest);
    \path[oriented edge] (a) .. controls (c4) .. (c3) .. controls (c2) and (c1) .. (e);        
    \path[oriented edge] (d) .. controls (c5) and (c6) .. (a);
    \path[oriented edge] (a) .. controls (c8) and (c7) .. (f);
\end{scope}
\begin{scope}[xshift=15cm, yshift=14cm]
\basegraph{0}{1}{1}{0}
    \foreach \source/\dest in {b/a,b/d,b/e,f/b,e/d,f/e}
    \path[oriented edge] (\source) -- (\dest);
    \path[oriented edge] (a) .. controls (c4) .. (c3) .. controls (c2) and (c1) .. (e);        
    \path[oriented edge] (a) .. controls (c6) and (c5) .. (d);
    \path[oriented edge] (f) .. controls (c7) and (c8) .. (a);
\end{scope}
\end{tikzpicture}
\end{center}
\caption{The distributive lattice induced by vertex flips in $c$-orientations. The reference orientation at the bottom is the directed dual $D^\ast$ of the orientation $D$ of the graph $G$ used in Figures~\ref{fig:match} and \ref{fig:flipgraph}, where some parallel arcs incident with $\top$ are grouped together for simplicity.
The numbers depicted at the vertices indicate the number of times that each vertex is flipped with respect to the reference orientation.}
\label{fig:lattice}
\end{figure}
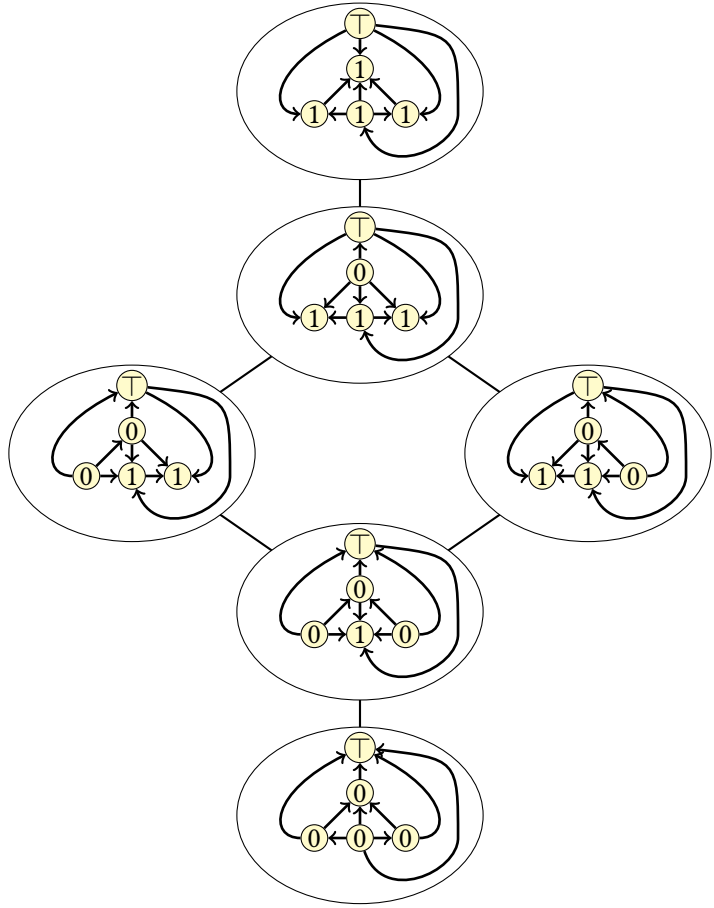

\paragraph{Every distributive lattice is a lattice of $c$-orientations.}

We next point out that every distributive lattice is isomorphic to the distributive lattice induced by a set of $c$-orientations of a graph.
This relationship was described by Knauer~\cite{K07b}.

In order to represent a given distributive lattice $L$ by an isomorphic lattice of $c$-orientations, we need to construct a corresponding digraph $D(L)$. For this purpose, we shortly recall a classical result from lattice theory, \emph{Birkhoff's Theorem} (see~\cite{BP06}).

For any distributive lattice $L$, $\cJ(L)$ is the subposet of $L$ induced by the set of \emph{join-irreducible} elements, these are the elements of $L$ covering exactly one element.
On the other hand, given any poset $P$ we may look at the distributive lattice $\cO(P)$ formed by the downsets of $P$ ordered by inclusion.
Birkhoff's Theorem in our setting asserts that those two operations  are inverse in the sense that $P \cong \cJ(\cO(P))$ for any finite poset $P$ and $\cO(\cJ(L)) \cong L$ for any finite distributive lattice.

The idea is to define a digraph $D(L)$ whose vertex set consists of the elements of $\cJ(L)$ with an additional vertex $\top$.
The digraph is obtained from the natural upward-orientation of $\cJ(L)$ plus additional arcs from all the sinks and sources to $\top$.
Let $G(L)$ be the underlying graph of $D(L)$ and let $L_c(D(L))$ denote the distributive lattice obtained from the set of $c$-orientations of $G(L)$ containing $D(L)$ with fixed (non-flippable) vertex $\top$.
An example of this construction is provided in Figure~\ref{fig:birkhoff}.

\begin{figure}
\centering
\includegraphics[scale=0.7]{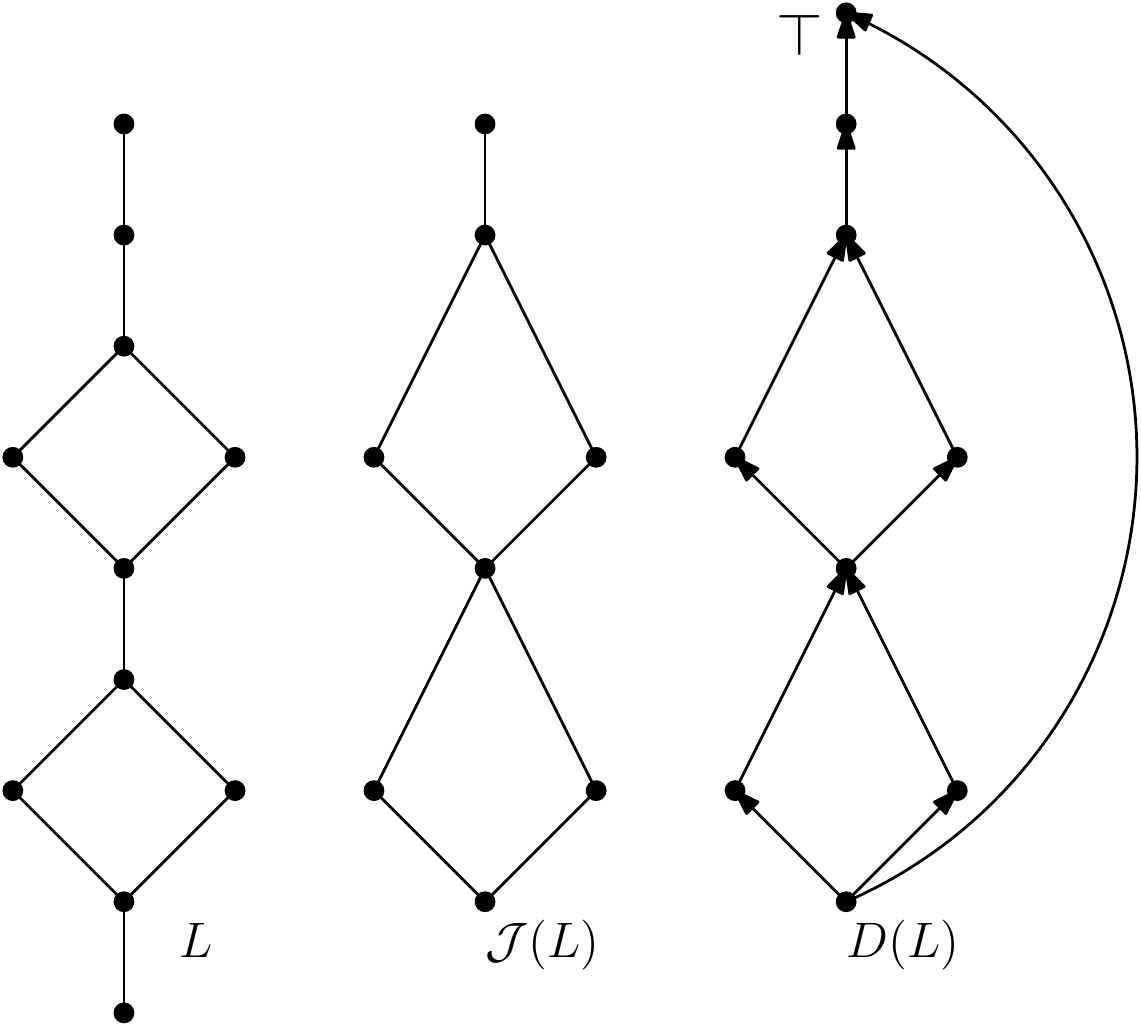}
\vspace{3mm}
\caption{A distributive lattice $L$ represented by its Hasse diagram (left), the corresponding subposet of join-irreducible elements $\cJ(L)$ (middle), and the digraph $D(L)$ associated with the lattice (right).}
\label{fig:birkhoff}
\label{poset}
\end{figure}

The following theorem is an easy consequence of Birkhoff's Theorem.

\begin{theorem}[\cite{K07b}]
\label{thm:universal}
Let $L$ be any distributive lattice and $D(L)$ be the corresponding digraph as defined above. Then $L \cong L_c(D(L))$.
\end{theorem}

Theorem~\ref{thm:embedding} below gives a natural geometric embedding of the lattice $L$ depending on the digraph $D(L)$.
This embedding is such that all values $z_X(x)$ are 0 or 1, and the vectors $z_X(.)$ are exactly the characteristic vectors of the downsets of $\cJ(L)$.
The convex hull of those vectors is known as the \emph{order polytope} of $\cJ(L)$~\cite{S86}, which is a particular case of the above-mentioned alcoved polytopes.
The problem of computing vertex flip distances between elements of $L$ encoded by $c$-orientations of $G(L)$ therefore boils down to computing the distance between two downsets of $\cJ(L)$ in their inclusion lattice, which is a simple special case of Problem~\ref{pb:vertexflip}.

\paragraph{Facial flips in planar graphs.}

When we consider Problem~\ref{pb:vertexflip} on planar graphs, restricting to vertex flips and considering the dual plane graph amounts to considering only flips of directed facial cycles, excluding the outer face whose dual vertex is $\top$.
We refer to these as \emph{facial flips}.
Felsner~\cite{F04} considered distributive lattices induced by facial flips.
The following computational problem is a special case of Problem~\ref{pb:vertexflip}.

\begin{problem}
\label{pb:faceflip}
Given a 2-connected plane graph $G$ and a pair $X,Y$ of strongly connected $\alpha$-orientations, what is the length of a shortest facial flip sequence transforming $X$ into $Y$?
\end{problem}

Zhang, Qian, and Zhang~\cite{ZQZ19} recently provided a closed formula for this flip distance, which can be turned into a polynomial-time algorithm.
We prove the analogous stronger statement for Problem~\ref{pb:vertexflip}.

\begin{theorem}
\label{thm:polynomial}
There is a polynomial-time algorithm that, given a graph $G$ with a fixed vertex $\top$ and a pair $X,Y$ of $c$-orientations of $G$, outputs a shortest vertex flip sequence between $X$ and $Y$.
\end{theorem}

In the planar case, this directly translates to a polynomial-time algorithm for Problem~\ref{pb:faceflip}.
The proof of Theorem~\ref{thm:polynomial} is presented in Section~\ref{sec:lattice}.
In~\cite{KF09}, the distributive lattice structure on $c$-orientations is generalized to so-called \emph{$\Delta$-bonds}, also known as \emph{tensions}. We believe that our proof of Theorem~\ref{thm:polynomial} can be generalized to these objects.

\paragraph{Flip distance with larger cut sets.}

While computing the cut flip distance between $c$-orientations is an \NP-hard problem in general (Theorem~\ref{thm:hardness}), there is a polynomial-time-algorithm for computing the distance when only using vertex flips (Theorem~\ref{thm:polynomial}).
It is natural to ask for a threshold between the hard and easy cases of flip distance problems.
Our hardness reduction in Section~\ref{sec:pm} involves very long directed cycles, which correspond to flips of directed cuts in the dual $c$-orientations with cut sets of large size.
Consequently, one may hope that the problem gets easier when restricting the sizes of the cut sets involved in a flip sequence.
Our last result destroys this hope:

\begin{theorem}
\label{thm:2flips}
Let $X,Y$ be $c$-orientations of a connected graph $G$ with fixed vertex $\top$.
It is \NP-hard to determine the length of a shortest cut flip sequence transforming $X$ into $Y$, which consists only of minimal directed cuts with interiors of size at most two.
\end{theorem}

We will present the proof of Theorem~\ref{thm:2flips} in Section~\ref{sec:jump}.

\section{Preliminaries}

In this section we recall some standard terminology concerning digraphs and posets that will be used repeatedly in the paper.

\paragraph{Cuts and cut sets.}
Given a directed graph $D$ and a subset $U \subseteq V(D)$ of vertices, we denote by $\delta(U)$ the set of all arcs in $E(D)$ having one end in $U$ and the other in $\overline{U}:=V(D) \setminus U$.
By $\delta^+(U)$, we denote the set of all arcs in $E(D)$ directed from $U$ to $\overline{U}$.
If $\delta(U)=\delta^+(U)$, we call $S=\delta(U)$ a \emph{directed cut} or \emph{dicut} induced by $U$, and $U$ is referred to as a \emph{cut set} of $S$.
In the case that $D$ is weakly connected, the cut set is uniquely determined by the dicut.

\paragraph{Posets and lattices.}
A \emph{partially ordered set}, or \emph{poset} for short, is a pair $(P,\prec)$, where $P$ is a set and $\prec$ is a reflexive, antisymmetric and transitive binary relation on $P$.
Posets can be represented more compactly by their minimal comparabilities: We say that $x\prec y$ is a \emph{cover relation}, or $y$ \emph{covers} $x$, if there is no $z$ in the poset with $x \prec z \prec y$.
This defines the \emph{cover graph} of $P$, which has the elements of $P$ as vertices, and an edge for every cover relation.
The \emph{Hasse diagram} of $P$ is a drawing of the cover graph in the plane, where vertices are represented by distinct points and for every cover relation $x\prec y$, the edge between $x$ and $y$ is drawn as a straight line going upwards from $x$ to $y$.

The \emph{downset} of an element $y$ in $P$ is the set of all $x$ with $x\prec y$.
A poset $P$ is called a \emph{lattice}, if for any two elements $x$ and $y$ in $P$ there is a unique smallest element $z$ such that $x\prec z$ and $y\prec z$, and a unique largest element $z$ such that $z\prec x$ and $z\prec y$.
These elements are called the \emph{join} and the \emph{meet} of $x$ and $y$, respectively.
A lattice is called \emph{distributive}, if the join and meet operations distribute over each other.

\section{Flip distance between perfect matchings and between $\alpha$-orientations}
\label{sec:pm}

The proof of Theorem~\ref{thm:hardness} is by reduction from the following \NP-complete problem.

\begin{theorem}[Ples\'{n}ik~\cite{P79}]
\label{thm:ples}
Deciding directed Hamiltonicity of orientations of cubic planar graphs without sinks and sources is \NP-complete.
\end{theorem}

It is clear that the above problem remains \NP-complete if we additionally assume 2-connectivity of the cubic graph.

\begin{proof}[Proof of Theorem~\ref{thm:hardness}]
As each flip sequence of length at most two can be used as a polynomially verifiable certificate, the problem is clearly in \NP.

We now provide a reduction of the decision problem in Theorem~\ref{thm:ples} to Problem~\ref{pb:bipmatchingdist}.
So suppose we are given an orientation $D$ of a 2-connected cubic planar graph without sinks and sources, and assume without loss of generality that $|V(D)|\ge 3$.
Given $D$, we define an undirected graph $G=G(D)$ as follows; see Figure~\ref{fig:gadgets}:
For each vertex $v\in V(D)$ we create a vertex $x_v$ in $G$, and for each arc $e \in E(D)$ we create a pair of vertices $x^+_e,x^-_e$ in $G$.
The edges of $G$ are defined as follows:
For each arc $e\in E(D)$, we connect $x^+_e$ and $x^-_e$ with an edge in $G$.
Furthermore, we denote by $V_1$ and $V_2$ the vertices of $D$ with outdegree 1 or 2, respectively.
For each  $v\in V_1$, if $e,f\in E(D)$ are the two incoming arcs at $v$ and $g$ is the outgoing arc, then we add the edges $x^+_e x_v$, $x^+_f x_v$, $x^+_e x^-_g$, and $x^+_f x^-_g$ to $G$.
Similarly, for each $v\in V_2$, if $e,f\in E(D)$ are the two outgoing arcs at $v$ and $g$ is the incoming arc, then we add the edges $x^-_ex_v$, $x^-_f x_v$, $x^-_e x^+_g$, and $x^-_f x^+_g$ to $G$.
We refer to the 4-cycles in $G$ formed by these edges as \emph{$C_4$-gadgets}.
Note that $G$ is subcubic, planar, 2-connected, and bipartite.
Specifically, the bipartition is given by $\{x_v\mid v\in V_1\}\cup\{x^-_e\mid e\in E(D)\}$ and $\{x_v\mid v\in V_2\}\cup\{x^+_e\mid e\in E(D)\}$.

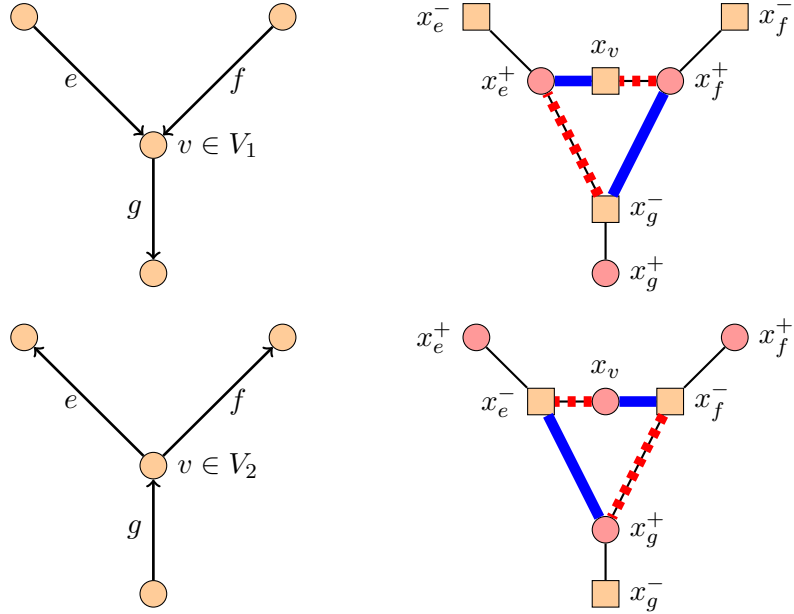
\begin{figure}[h]
\begin{center}
\begin{tikzpicture}[scale=.85]
  \begin{scope}[yshift=5cm]
    \foreach \pos/\name in {{(1,5)/a},{(3,1)/b},{(5,5)/c}}
        \node[vertex] (\name) at \pos {};
    \node[vertex, label=right:{$v\in V_1$}] (d) at (3,3) {};
    \path[oriented edge] (a) -- node[anchor=east] {$e$} (d);
    \path[oriented edge] (c) -- node[anchor=west] {$f$} (d);
    \path[oriented edge] (d) -- node[anchor=east] {$g$} (b);
  \end{scope}
  \begin{scope}
    \foreach \pos/\name in {{(1,5)/a},{(3,1)/b},{(5,5)/c}}
        \node[vertex] (\name) at \pos {};
    \node[vertex, label=right:{$v\in V_2$}] (d) at (3,3) {};
    \path[oriented edge] (d) -- node[anchor=east] {$e$} (a);
    \path[oriented edge] (d) -- node[anchor=west] {$f$} (c);
    \path[oriented edge] (b) -- node[anchor=east] {$g$} (d);    
  \end{scope}
  \begin{scope}[xshift=7cm,yshift=5cm]
    \node[even vertex, label=left:{$x^+_e$}] (e) at (2,4) {};
    \node[even vertex, label=right:{$x^+_f$}] (g) at (4,4) {};
    \node[even vertex, label=right:{$x^+_g$}] (b) at (3,1) {};
    \node[odd vertex, label=left:{$x^-_e$}] (a) at (1,5) {};
    \node[odd vertex, label=right:{$x^-_f$}] (c) at (5,5) {};
    \node[odd vertex, label=right:{$x^-_g$}] (d) at (3,2) {};
    \node[odd vertex, label=above:{$x_v$}] (f) at (3,4) {};    
    \foreach \dest / \source in {a/e,e/d,d/b,e/f,f/g,d/g,g/c}
        \path[edge] (\source) -- (\dest);
    \foreach \dest / \source in {e/f,d/g}
        \path[matched edge] (\source) -- (\dest);
    \foreach \dest / \source in {f/g,e/d}
        \path[other matched edge] (\source) -- (\dest);
  \end{scope}
  \begin{scope}[xshift=7cm]
    \node[even vertex, label=left:{$x^+_e$}] (a) at (1,5) {};
    \node[even vertex, label=right:{$x^+_f$}] (c) at (5,5) {};
    \node[even vertex, label=right:{$x^+_g$}] (d) at (3,2) {};
    \node[even vertex, label=above:{$x_v$}] (f) at (3,4) {};
    \node[odd vertex, label=left:{$x^-_e$}] (e) at (2,4) {};
    \node[odd vertex, label=right:{$x^-_f$}] (g) at (4,4) {};
    \node[odd vertex, label=right:{$x^-_g$}] (b) at (3,1) {};
    \foreach \dest / \source in {a/e,e/d,d/b,e/f,f/g,d/g,g/c}
        \path[edge] (\source) -- (\dest);
    \foreach \dest / \source in {e/f,d/g}
        \path[other matched edge] (\source) -- (\dest);
    \foreach \dest / \source in {f/g,e/d}
        \path[matched edge] (\source) -- (\dest);
  \end{scope}
\end{tikzpicture}
\vspace{3mm}
\end{center}
\caption{$C_4$-gadgets to construct the undirected graph $G=G(D)$ (right) from the digraph $D$ (left) in the proof of Theorem~\ref{thm:hardness}.
The edges of the matchings $X$ and $Y$ in $G$ are indicated by bold solid and dashed lines, respectively.}
\label{fig:gadgets}
\end{figure}

We construct a pair of perfect matchings $X,Y$ on $G$ as follows.
The first matching $X$ is defined by fixing a particular perfect matching on each $C_4$-gadget, and the second matching $Y$ is obtained from $X$ by flipping all cycles formed by the $C_4$-gadgets; see Figure~\ref{fig:gadgets}.
We claim that $X$ and $Y$ have flip distance at most two in $G$ if and only if $D$ has a directed Hamiltonian cycle.
From this the theorem follows.

First, assume there is a directed Hamiltonian cycle $H$ in $D$.
We define a pair of cycles $C_1,C_2$ in $G$, where 
$C_1$ and $C_2$ both contain all the edges $\{x^+_e x^-_e\mid e \in E(H)\}$, plus additional edges defined as follows.
For each vertex $v \in V(D)$, consider the corresponding $C_4$-gadget in $G$ with the two incident edges corresponding to the edges incident with $v$ on $H$.
The endpoints of those edges on the gadget divide it into two alternating paths in $X$, one with matching edges at both ends and one with non-matching edges at both ends. We add the edges on those two types of paths to $C_1$ or $C_2$, respectively.
Note that $C_1$ is an alternating cycle in $X$.
Moreover, after flipping $C_1$, the cycle $C_2$ is alternating, and flipping $C_2$ yields $Y$, as each edge in a $C_4$-gadget gets flipped once while the remaining edges are flipped an even number of times and thus remain unchanged.

For the reverse implication, assume that $X$ and $Y$ are connected by flipping at most two alternating cycles.
As the symmetric difference $X\Delta Y$ contains at least three disjoint cycles (recall the assumption $|V(D)| \ge 3$), exactly two cycles $C_1,C_2$ are flipped to transform $X$ into $Y$, and neither $C_1$ nor $C_2$ is one of the 4-cycles formed by the $C_4$-gadgets.
As the edges outside the gadgets remain unchanged, they are covered by both $C_1,C_2$ or by neither of them.
We claim that $H:=\{e \in E(D)\mid x^+_e x^-_e \in E(C_1)\}$ is the arc set of a directed Hamiltonian cycle in $D$.
Since up to isomorphism, $H$ is obtained from $E(C_1)$ by contraction of the $C_4$-gadgets, $H$ forms a cycle in $D$ (here we need that $D$ is cubic).
If the cycle $H$ is not a directed cycle in $D$, there would be some $v \in V_1$ with two incoming incident edges from $H$.
However, in this case, the path in the corresponding $C_4$-gadget contained in $C_1$ consists of two edges, one of which is not in $X$, contradicting that $C_1$ is an alternating cycle in $X$.
Finally, the directed cycle $H$ has to be spanning.
Indeed, if there is a $C_4$-gadget not covered by $C_1$, then $C_2$ would have to cover exactly this gadget, a contradiction.
\end{proof}

\begin{proof}[Proof of Theorem~\ref{thm:hardnessdiff}]
We use the following hardness result of Peroche~\cite{Per84}.  Given a digraph $D$, where each vertex has indegree and outdegree equal to $2$, it is \NP-complete to decide if $E(D)$ is the union of two directed Hamiltonian cycles.  Given such a digraph $D$, let $\overleftarrow{D}$ be the digraph obtained from $D$ by reversing the direction of every arc.  We regard $D$ and $\overleftarrow{D}$ as $\alpha$-orientations $X$ and $Y$ of the same underlying graph $G$, where $\alpha(v)=2$ for all $v \in V(G)$. The theorem follows by observing that the flip distance between $X$ and $Y$ is at most $2$ if and only if $E(D)$ is the union of two directed Hamiltonian cycles. Moreover, the same statement holds when we only allow flipping edges that are oriented differently in $X$ and $Y$. 
\end{proof}

\section{Vertex flip distance between $c$-orientations}
\label{sec:lattice}

In this section we prove Theorem~\ref{thm:polynomial}.

Recall that, given a graph $G$ with a fixed vertex $\top$, we only allow vertex flips at vertices distinct from $\top$.
In the case that $G$ is connected, we distinguish between two types of dicuts as follows: we say that a dicut $S$ in an orientation of $G$ is \emph{positive} with respect to $\top$ if and only if the uniquely determined cut set $U$ of $S$ does not contain $\top$.
Otherwise the dicut is called \emph{negative}.
We also define the \emph{interior} of $S$, denoted $\Int(S)$, as the cut set $U$ of $S$ if $S$ is positive and as its complement $\overline{U}$ if $S$ is negative. That is, $\Int(S)$ is the set of vertices on the side of the cut opposite to $\top$.

The following lemma is needed to decompose the edges of certain digraphs into dicuts with nested cut sets.
Formally, for a digraph $D$ and dicuts $S_1=\delta(U_1)$ and $S_2=\delta(U_2)$, the pair $S_1,S_2$ is called \emph{laminar} if either $U_1 \subseteq U_2$, $U_2 \subseteq U_1$, $U_1 \cap U_2=\emptyset$, or $U_1 \cup U_2=V(D)$; see Figure~\ref{fig:poset}.
A family of dicuts in $D$ is called \emph{laminar} if all of its pairs are laminar.
A \emph{balanced digraph} is a digraph in which every cycle of the underlying graph has the same number of forward and backward edges.

\begin{lemma}
\label{lem:laminar}
Let $D$ be a balanced digraph.
Then $E(D)$ can be decomposed into a laminar family of disjoint minimal dicuts.
\end{lemma}

\begin{proof}
We prove the statement by induction on $|E(D)|$.
It is clearly true if $E(D)=\emptyset$.
Assume for the induction step that $|E(D)|=k \ge 1$ and the statement holds for all digraphs with less than $k$ arcs.

As $D$ is balanced, it is obviously acyclic and therefore contains a source $s \in V(D)$.
Each cycle in the underlying graph of $D$ that contains $s$ has exactly one forward and one backward edge incident to $s$.
Therefore, the digraph $D'$ obtained from $D$ by contracting the set $\delta^+(\{s\})$ of arcs incident to $s$ is still balanced and has less than $k$ edges.
By the induction hypothesis, there exists a laminar decomposition of $E(D')=E(D) \setminus \delta^+(\{s\})$ into disjoint dicuts in $D'$ and thus in $D$.
Note that the dicuts in $D'$ are exactly those of $D$ disjoint from $\delta^+(\{s\})$, and laminarity is preserved.
Hence adding a decomposition of the directed vertex cut $\delta^+(\{s\})$ into minimal dicuts to the collection gives rise to a decomposition of $E(D)$ into disjoint minimal dicuts.
This resulting decomposition is also laminar, as the cut set of the new minimal cuts is $\{s\}$, which is either contained in or disjoint from each cut set of the original decomposition of $E(D')$.
\end{proof}

For every pair $X,Y$ of $c$-orientations on a graph $G$, the \emph{difference} $X \setminus Y$ denotes the set of arcs in $X$ whose orientation is reversed in $Y$.
The difference forms a balanced subdigraph of $X$ (and reversing them yields a balanced subdigraph of $Y$).
Consequently, Lemma~\ref{lem:laminar} provides another proof that $c$-orientations can be reached from one another by flipping minimal dicuts.

We now consider the partial order defined on acyclic $c$-orientations of an $n$-vertex graph such that the covering relation corresponds to flipping a source vertex.
Recall that by Theorem~\ref{thm:DL}, this partial order is a distributive lattice.
We now reuse a result from Propp~\cite{P02} and Felsner and Knauer~\cite{KF09} that gives an embedding of this distributive lattice into $\mathbb{N}^{n-1}$, which led to the introduction of distributive polytopes by Felsner and Knauer~\cite{KF11}.
This theorem is illustrated in Figure~\ref{fig:lattice}, where the values of the functions $z_X$ are depicted in the vertices of the graph $G$.

\begin{theorem}[\cite{P02,KF09}]
\label{thm:embedding}
Let $G$ be a graph on $n$ vertices with a fixed vertex $\top$, $X$ an acyclic $c$-orientation of $G$, and denote by $X_{\min}$ the minimal element of the associated distributive lattice.
Then the number of times $z_X(x)$ a vertex $x \in V(G) \setminus \{\top\}$ is flipped in an upward lattice path from $X_{\min}$ to $X$ is independent of the sequence.
The resulting function $z_X:\cR_c\to\mathbb{N}^{n-1}$ is a lattice embedding. That is, for every $x,y\in\mathbb{N}^{n-1}$ corresponding to $c$-orientations of $G$, the join and meet correspond to $\min(x,y)$ and $\max(x,y)$, respectively.
\end{theorem}

In other words, the distributive lattice on $\cR_c$ is isomorphic to an induced sublattice of the componentwise dominance order on $\mathbb{N}^{n-1}$.
We call a vertex flip sequence \emph{monotone} if every flipped vertex is either only flipped as a source or only as a sink.
With this definition, Theorem~\ref{thm:embedding} yields the following:

\begin{corollary}
\label{cor:monotone}
Let $G$ be a graph with fixed vertex $\top$ and $X,Y$ a pair of acyclic $c$-orientations on $G$.
Then every monotone vertex flip sequence transforming $X$ into $Y$ has minimal length.
\end{corollary}

Consider two $c$-orientations $X, Y$ of $G$.
Our goal is to construct a monotone flip sequence from $X$ to $Y$.
By Lemma~\ref{lem:laminar}, there is a laminar decomposition of the difference $X \setminus Y$ into minimal dicuts whose reversal yields $Y$.
Denote by $\cS=\cS(X,Y)$ the dicuts of one such decomposition.
We construct a poset $P$ on $\cS$ by the inclusion order of the interiors of the minimal dicuts. That is, for $S,T\in\cS$, $S$ is ordered before $T$ in $P$ if and only if $\Int(S) \subseteq \Int(T)$; see Figure~\ref{fig:poset}.
Since $\cS$ is laminar, the cover graph of $P$ is a forest, with the additional property that every non-maximal element $S$ is covered by a unique other element, which we denote by $\cov(S)$.
Moreover, for each vertex $x \in V(G)$ in the interior of at least one of the cuts in $\cS$, we let $S_x$ be the (unique) minimal element of the poset $P$ such that $x\in \Int(S_x)$.
Also, for each $S \in \cS$ we denote by $\underline{\Int}(S):=\{x \in V(G)\mid S_x=S\} \subseteq \Int(S)$ the set of vertices in the interior of $S$ but in none of the interiors of the cuts covered by $S$ in the poset.

For each dicut $S\in\cS$ we define an integer \emph{weight $w(S)$} and a \emph{sign $\sgn(S)\in\{+,0,-\}$} as follows; see Figure~\ref{fig:poset}.
If $S\in \cS$ is a maximal element in $P$ then we define $w(S):=1$, and $\sgn(S):=+$ if $S$ is positive and $\sgn(S):=-$ otherwise.
For every sign $s \in \{+,0,-\}$ and dicut $S \in \cS$, we say that $S$ \emph{agrees} with $s$, if either $s=0$, or if $s=+$ and $S$ is positive, or $s=-$ and $S$ is negative.
For every non-maximal $S \in \cS$, we inductively define
$$w(S):=\begin{cases} w(\cov(S))+1 & \text{ if }S \text{ agrees with }\sgn(\cov(S)), \cr w(\cov(S))-1 & \text{ otherwise}, \end{cases}$$
and
$$\sgn(S):=\begin{cases} \sgn(\cov(S)) & \text{if } \sgn(\cov(S)) \neq 0 \text{ and } w(S) \neq 0, \cr + & \text{if } \sgn(\cov(S)) = 0, w(S) \neq 0 \text{ and }S \text{ is positive}, \cr - & \text{if } \sgn(\cov(S)) = 0, w(S) \neq 0 \text{ and }S \text{ is negative}, \cr 0 & \text{if }w(S)=0. \end{cases}$$

\begin{figure}
\centering
	\includegraphics[scale=0.7]{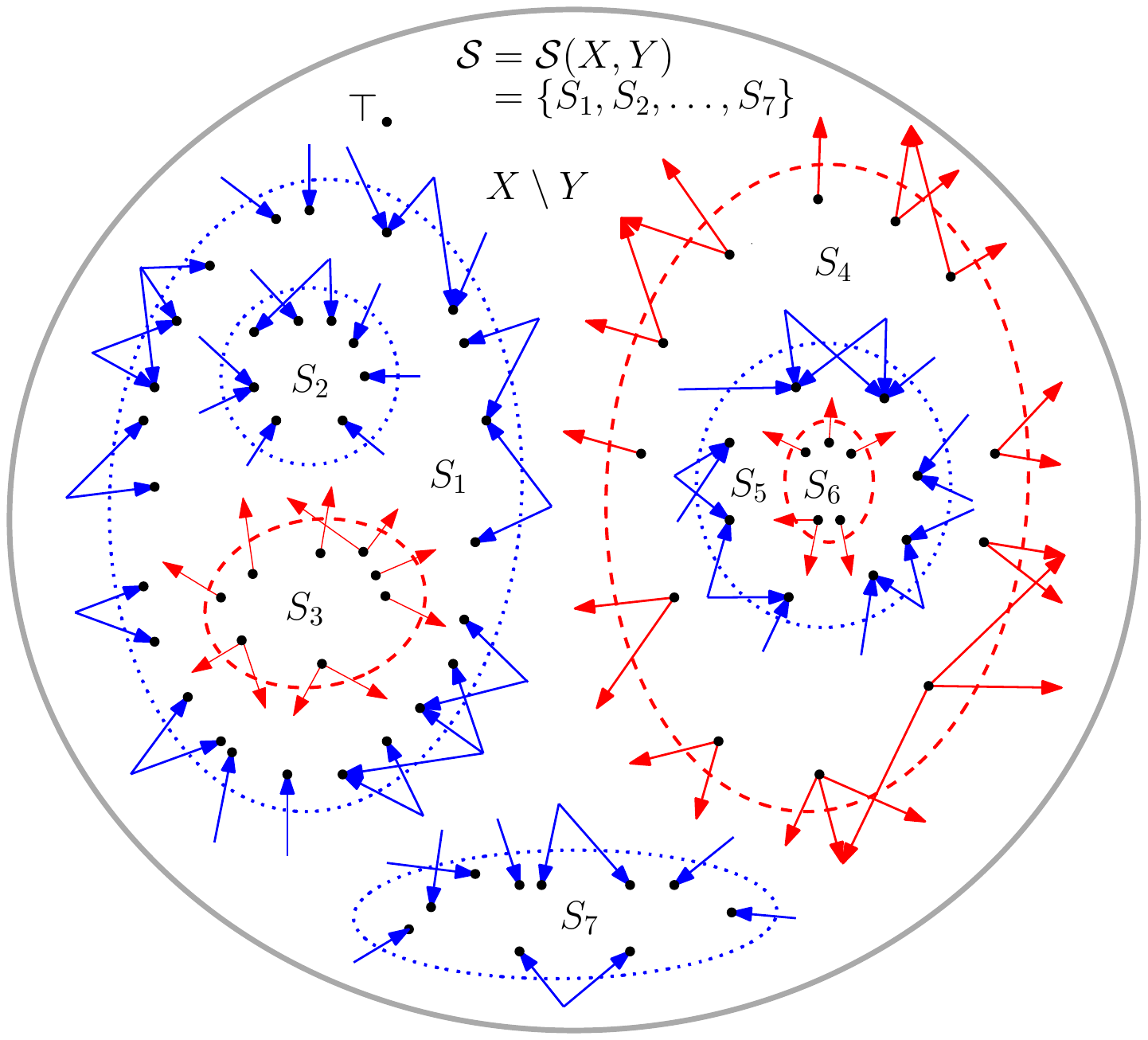}\hspace{0.4cm}
\includegraphics[scale=0.9]{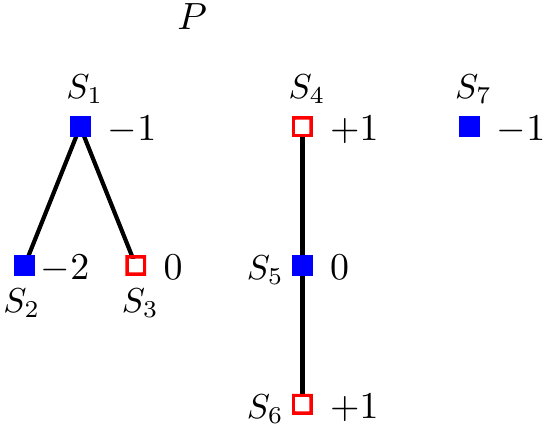}
\vspace{3mm}
\caption{A laminar collection $\cS$ of disjoint minimal dicuts of $X\setminus Y$ (left), where the positive ones are dashed, and the negative ones are dotted, and the corresponding poset $P$ of dicuts in $\cS$ ordered by inclusion (right) with its associated signed weights $\sgn(S)\cdot w(S)$, $S\in\cS$.}
\label{fig:poset}
\end{figure}

It follows from this definition that the weights are non-negative and that $\sgn(S)=0$ if and only if $w(S)=0$ for every $S \in \cS$.
We will see that given a minimal dicut $S$ in $\cS$, the weight $w(S)$ describes the number of times each vertex which lies in $\underline{\Int}(S)$ will be flipped, whereas $\sgn(S)$ captures the direction in which (all) these vertices are flipped.  That is, a positive sign means that vertices are flipped from sources to sinks, while a negative sign means that vertices are flipped from sinks to sources.
We will need the following auxiliary statement.

\begin{lemma}
\label{lem:flip-int}
Let $X, Y$ be acyclic $c$-orientations of a connected graph $G$ with fixed vertex $\top$.
If $S=X\setminus Y$ is a positive dicut, then there is a vertex flip sequence transforming $X$ into $Y$ such that only vertices in $\Int(S)$ are flipped, each exactly once from source to sink.

The analogous statement for negative dicuts holds with sources and sinks exchanged.
\end{lemma}

\begin{proof}
We prove the statement by induction on $|\Int(S)|$. If $\Int(S)$ is a single vertex, then $S$ corresponds to the arcs incident to a source, and the statement holds.

For the induction step assume $|\Int(S)|=k \ge 2$ and that the claim holds for all positive cuts in $c$-orientations of $G$ whose interiors have size less than $k$.
Since $X$ is acyclic, the induced subdigraph $X[\Int(S)]$ is also acyclic and thus contains a source $x \in \Int(S)$.
Since $\Int(S)$ is the cut set of $S$, $x$ is a source in $X$ as well.
Let $Z$ be the $c$-orientation obtained from $X$ by a vertex flip at $x$.
It follows that the cut $\delta(\Int(S) \setminus \{x\})$ in $Z$ is positive with interior of size $k-1$.
By induction, there is a vertex flip sequence from $Z$ to $Y$ such that only vertices in $\Int(S) \setminus \{x\}$ are flipped, each exactly once from source to sink.
Starting with a vertex flip at $x$ and continuing with this flip sequence yields a flip sequence from $X$ via $Z$ to $Y$ with the desired properties.
\end{proof}

We are now in position to prove the main result of this section.

\begin{theorem}
\label{thm:main}
Let $X,Y$ be acyclic $c$-orientations of a connected graph $G$.
There is a monotone vertex flip sequence transforming $X$ into $Y$ which can be computed in time polynomial in $|E(G)|$.
\end{theorem}

\begin{proof}
Consider the following strengthening of the theorem:

{\it \noindent \underline{Claim:} Let $X,Y$ be acyclic $c$-orientations of a connected graph $G$ and $\cS=\cS(X,Y)$ a laminar decomposition of $X \setminus Y$ into disjoint minimal dicuts.
Then there is a monotone vertex flip sequence from $X$ to $Y$, such that every flipped vertex $x$ is contained in the interior of some dicut of $\cS$, and $x$ is flipped $w(S_x)$ times from source to sink if $\sgn(S_x)=+$, and from sink to source otherwise.
}

We prove this claim by induction on the size of $\cS$.
The statement is clearly true if $X=Y$ (which means that $|\cS|=0$), settling the base case of the induction.
Assume for the induction step that we are given a pair $X \neq Y$ of $c$-orientations and a laminar decomposition $\cS$ of $X \setminus Y$ of size $k \ge 1$.
Assume that the claim holds for all pairs of $c$-orientations with a laminar decomposition of size less than $k$.

In the poset $P$ on $\cS$ we consider a minimal element corresponding to a cut $S\in\cS$, i.e., we have $\underline{\Int}(S)=\Int(S)$ and all vertices $x\in\Int(S)$ satisfy $S_x=S$.
Lemma~\ref{lem:flip-int} gives a vertex flip sequence $F_1$ that flips only vertices in $\Int(S)$, each exactly once from source to sink if $S$ is positive and from sink to source if $S$ is negative.
Applying this flip sequence to $X$, we obtain an intermediate $c$-orientation $Z$ that differs from $X$ only by the reversal of all edges in $S$.
Consequently, $\cS \setminus \{S\}$ is a laminar decomposition of $Z \setminus Y$ into minimal dicuts in $Z$ of size $k-1$.
By induction, we also have a vertex flip sequence $F_2$ transforming $Z$ into $Y$ with the aforementioned properties.

Note that the weights and signs of all dicuts $T\in\cS\setminus\{S\}$ defined with respect to $\cS$ or $\cS\setminus\{S\}$ are the same, so we may simply write $w(T)$ and $\sgn(T)$.
Furthermore, the set $\underline{\Int}(T)$ defined with respect to $\cS$ is a subset of the same set defined with respect to $\cS \setminus \{S\}$.
To complete the induction step, we distinguish two cases.

The first case is that $S$ is a maximal element in $P$, or that $S$ agrees with $\sgn(\cov(S))$.
In this case, we claim that  the concatenation $F$ of $F_1$ and $F_2$ is a flip sequence transforming $X$ via $Z$ into $Y$ with the desired properties.
It suffices to check this for the vertices in $\underline{\Int}(S)=\Int(S)$, since for all other vertices, the claimed properties follow inductively (they are never flipped in $F_1$, so their behavior in $F$ will be the same as in $F_2$).
If $S$ is a maximal element in $P$, then $w(S)=1$ and every vertex $x\in\Int(S)$ will be flipped exactly once.
Moreover, according to Lemma~\ref{lem:flip-int}, if $S$ is positive, i.e., $\sgn(S)=+$, then $x$ is flipped from source to sink, and if $S$ is negative and $\sgn(S)=-$, then $x$ is flipped sink to source.
It remains to consider the subcase that $S$ is not maximal, i.e., $\cov(S)$ exists.
Consider any vertex $x\in \Int(S)$.
During the flip sequence $F$, the vertex $x$ is flipped once in $F_1$ and $w(\cov(S))$ times in $F_2$, so $w(S)=w(\cov(S))+1$ times in total, as required.
Moreover, the assumption that $S$ agrees with $\sgn(\cov(S))$ means that either $\sgn(\cov(S))=+$ and $S$ is positive or $\sgn(\cov(S))=-$ and $S$ is negative, or $w(\cov(S))=\sgn(\cov(S))=0$.
We conclude from the inductive assumption that in those three cases, $x$ is only flipped from source to sink in both $F_1$ and $F_2$, or only sink to source in both, or only once in $F_1$ but not in $F_2$, respectively.
Consequently, $x$ satisfies the inductive claim in all cases.

The second case is that $S$ is not maximal, i.e., $\cov(S)$ exists, and that $S$ does not agree with $\sgn(\cov(S))$.
This means that $w(S)=w(\cov(S))-1$.
Without loss of generality, assume that $S$ is positive and consequently $\sgn(\cov(S))=-$ (the other case is symmetric).
Consider again the vertex flip sequence $F$ obtained by concatenating $F_1$ and $F_2$.
This flip sequence would transform $X$ via $Z$ into $Y$, however, we will not actually apply $F$, but modify the sequence as follows.
By Lemma~\ref{lem:flip-int}, $F_1$ flips each vertex in $\Int(S)$ exactly once from source to sink.  By induction, in $F_2$, each vertex in $\Int(S)$ contained in the interior of $\cov(S)$ (defined with respect to $\cS\setminus\{S\}$) is flipped from sink to source.
Let $x$ be the last element of $F_1$ and consider the subsequence $x,x_1,\ldots,x_k,x$ of $F$ starting with $x$ and ending with the first occurrence of $x$ in $F_2$.
None of the vertices $x_1,\ldots,x_k$ is adjacent to $x$ in $G$, because after the first vertex flip at $x$ (from source to sink) all edges incident with $x$ are incoming, and in $F_2$ we only flip sinks to sources.
This shows that deleting the first two occurrences of $x$ from $F$ preserves the number of and direction of all flips at vertices distinct from $x$, and still transforms $X$ into $Y$.
Repeated application of this argument produces a reduced vertex flip sequence $F'$ transforming $X$ into $Y$ such that each vertex $x \in V(G) \setminus \Int(S)$ is flipped the same number of times and in the same direction as in $F_2$.
That is, $x$ is flipped $w(S_x)$ times from source to sink if $\sgn(S_x)=+$, and $w(S_x)$ times from sink to source if $\sgn(S_x)=-$.
On the other hand, every $x \in \Int(S)$ is missing its first occurrence but is flipped in the same way from sink to source for all remaining occurrences.
This implies that $x$ is flipped $w(S)=w(\cov(S))-1$ times from sink to source, as it should.
This proves that $F'$ is a vertex flip sequence from $X$ to $Y$ satisfying the conditions in our claim, completing its proof.

It remains to verify that the recursive algorithm obtained from this inductive argument runs in time polynomial in $m:=|E(G)|$.
First of all, the number of dicuts in any laminar decomposition, which corresponds to the number of induction steps, is bounded by the number of edges $m$.
Consequently, it suffices to show that the number of operations needed in one induction step is bounded by a polynomial in $m$.
Specifically, we need to compute the cover relations of $P$, the weights and signs of the dicuts, find a minimal element of the poset $P$, test its properties for the case distinction and construct the resulting flip sequence by concatenation and possibly deletion of double occurrences, all of which can be done in time $\cO(m^2)$.
This proves an upper bound of $\cO(m^3)$ for the total number of steps performed for the construction of the monotone flip sequence to transform $X$ into $Y$.
Finally, a laminar decomposition $\cS$ of $X\setminus Y$ as guaranteed by Lemma~\ref{lem:laminar} can be computed in polynomial time, by following the recursive strategy explained in the proof of the lemma.
This completes the proof.
\end{proof}

Combining Corollary~\ref{cor:monotone} and Theorem~\ref{thm:main} yields Theorem~\ref{thm:polynomial}.

\section{Flip distance with larger cut sets}
\label{sec:jump}

In this section we prove Theorem~\ref{thm:2flips} by reduction from the following \NP-hard problem.

Given a (finite) poset $(P,\prec)$, its \emph{height} is the maximum size $k$ of a chain $x_1\prec x_2\prec \cdots \prec x_k$ in $P$.
A \emph{linear extension} of $P$ is a sequence $(x_1,\ldots,x_n)$ of all elements of $P$ such that $x_i\prec x_j$ implies that $i<j$.
Given a linear extension $L=(x_1,\ldots,x_n)$ of $P$, a \emph{jump} is a pair $x_i,x_{i+1}$ in $L$ for which $x_i \not\prec x_{i+1}$ in $P$.
Conversely, a \emph{bump} is a pair $x_i,x_{i+1}$ such that $x_i \prec x_{i+1}$.
The \emph{jump number} $s(P)$ of $P$ is the minimum number of jumps among all linear extensions of $P$.
The \emph{Jump Number Problem} is the algorithmic problem of computing the jump number of a poset given by its comparabilities.

\begin{theorem}[\cite{P81,M90}]
\label{thm:jump}
Determining the jump number of a poset of height two is \NP-hard.
\end{theorem}

\begin{proof}[Proof of Theorem~\ref{thm:2flips}]
We provide a Turing-reduction of the Jump Number Problem for posets of height two to the problem stated in the theorem.
For this purpose, assume we are given a poset $(P,\prec)$ of height two with bipartite Hasse diagram $G=(P_1\cup P_2, E)$ as an instance for the Jump Number Problem.
We may assume that $P$ has no isolated elements and that $P_1$ contains all minimal elements and $P_2$ all maximal elements of the poset.
We construct an auxiliary Hasse diagram $G'$ from $G$ by adding an additional unique maximal element $\top$, and connecting it with edges to all vertices of $G$.
We construct two orientations $X,Y$ of $G'$ as follows:
In both orientations all edges are oriented from $P_1$ to $P_2$.
Moreover, in $X$ all edges incident with $\top$ are oriented towards $\top$, while in $Y$ all these edges are oriented away from $\top$.
As $X,Y$ are obtained from each other by flipping all edges incident with $\top$ (this flip is not allowed, though, as $\top$ is the fixed vertex), they are $c$-orientations with respect to the same~$c$.

Let $d$ denote the minimal flip distance between these $c$-orientations according to the conditions of the theorem.
We will complete the proof by showing that $s(P)=d-1$.

We first show that $s(P)\ge d-1$.
For this argument, let $L=(x_1,\ldots,x_n)$ be an arbitrary linear extension of $P$.
As $P$ has height two, the elements $\{x_1,\ldots,x_n\}$ of $P$ are partitioned into subsets $B_1,\ldots,B_m$ of size one or two, such that for all $B_i,B_j$ with $i<j$, the elements from $B_i$ appear before the elements from $B_j$ in $L$, and such that the two-element sets $B_i$ contain exactly all bump pairs.
We define a flip sequence that starts with the orientation $X$ and consecutively flips the cuts induced by $B_1,\ldots B_m$.
Since for all $1\le i\le m$, $B_i$ and $\overline{B_i}:=(P\cup\{\top\})\setminus B_i$ induce connected subgraphs of $G'$, these are indeed minimal cuts.
Moreover, each of these cuts is flippable.
This is obviously true for $B_1$, as $B_1$ induces a dicut in $X$.
Now assume inductively that the cuts induced by $B_1,\ldots,B_{k-1}$ for some $k \ge 2$ have been flipped.
As $L$ is a linear extension of $P$, all elements in the downset of $B_k$ in $P$ but not in $B_k$ are in one of the $B_i$ with $i<k$.
This implies that every arc between some $x \in B_k$ and $y \notin B_k$ is oriented from $x$ to $y$ in the current orientation, and thus $B_i$ is indeed flippable.

In this flip sequence, every arc in $X$ not incident to $\top$ will be flipped zero or two times and thus maintains its original orientation, while all the edges incident to $\top$ get reversed, as they are incident to exactly one set $B_i$.
Consequently, the flip sequence transforms $X$ into $Y$, proving that $d\le m$.
As $m$ equals the number of jumps in $L$ plus 1 (every non-jump is a bump within one of the $B_i$), this yields $d-1\le s(P)$.

We now show that $s(P)\le d-1$.
Assume that $B_1,\ldots,B_m \subseteq P$ are the cut sets of size one or two appearing (in this order) in a shortest flip sequence transforming $X$ into $Y$.
We may assume that among all shortest flip sequences, this sequence also minimizes $|B_1|+|B_2|+\ldots+|B_m|$.
Since each vertex $x \in P$ has an outgoing arc to $\top$ in $X$ which must be reversed during the flip sequence, $x$ must be contained in at least one of the $B_i$.
We claim that $x$ is contained in at most one of the $B_i$.  That is, the $B_i$ are pairwise disjoint.
Assume to the contrary that $x \in B_i \cap B_j$ for some $i<j$ and that $B_i,B_j$ is the only intersecting pair among $B_i,B_{i+1},\ldots,B_j$ (by minimizing $j-i$).
In particular, none of the cut sets $B_{i+1},\ldots,B_{j-1}$ contains $x$, and $x$ is the only vertex flipped multiple times in this subsequence.
We are then in one of the four cases $B_i=B_j=\{x\}$, or $B_i=\{x,y\}$ and $B_j=\{x\}$, or $B_i=\{x\}$ and $B_j=\{x,z\}$, or $B_i=\{x,y\}$ and $B_j=\{x,z\}$ for some elements $y,z\in P$ distinct from $x$.
Since no vertex adjacent to $x$ in $G'$ can be flipped by $B_{i+1},\ldots,B_{j-1}$, it follows that in each of these cases, the sequence
$$B_1,\ldots,B_{i-1},B_i\setminus \{x\},B_{i+1},\ldots,B_{j-1},B_j\setminus \{x\},B_{j+1},\ldots,B_m$$
is a valid flip sequences from $X$ to $Y$ of length at most $m$ and with decreased sum $|B_1|+|B_2|+\ldots+|B_m|$, a contradiction.
This proves that the cut sets $B_i$ are pairwise disjoint.

The $B_i$ are flipped one after the other and by definition of $X$, the dicut induced by $B_i$ is flippable if and only if all the elements in the downset of $B_i$ with respect to $P$ but not in $B_i$ were flipped before.
Therefore, by listing the elements in the sets $B_1,\ldots,B_m$ in this relative order, and ordering the elements within each $B_i$ according to their order in $P$, we obtain a linear extension $L$ of $P$ whose jumps are exactly those pairs having elements in two consecutive sets $B_i$.
It follows that there are $m-1=d-1$ jumps in $L$, proving that $s(P) \le d-1$.

Combining these arguments shows that $s(P)=d-1$, and using Theorem~\ref{thm:jump} we obtain the claimed hardness result.
\end{proof}

\section{Open problems}

Recall that Problem~\ref{pb:alphadistdiff} asks for a shortest flip sequence of directed cycles transforming one $\alpha$-orientation $X$ into another one $Y$, where we only allow flipping edges that are oriented differently in $X$ and $Y$.
Since the set of edges that are oriented differently in $X$ and $Y$ form a Eulerian subdigraph $D$ of both $X$ and $Y$, we have the following natural question:

\begin{question}
\label{quest:decompose}
What is the smallest number of directed cycles into which a Eulerian digraph can be decomposed?
\end{question}

We have seen in Theorem~\ref{thm:hardnessdiff} that from a computational point of view, this problem is hard for general digraphs, but we wonder what happens when adding planarity constraints.
The aforementioned question can also be studied in terms of upper bounds as a function of the number of vertices, which is related to the famous Haj\'os conjecture on undirected Eulerian graphs, see~\cite{Lov-68}.
Another interesting undirected variant of Question~\ref{quest:decompose} is the following:

\begin{question}
\label{quest:undirected}
Given a graph $G$ with a Eulerian subgraph $H$.
What is the smallest number of cycles of $G$ such that their symmetric difference is $H$?
\end{question}

Concerning our proof of Theorem~\ref{thm:2flips}, we believe that for any bound on the size of the cuts, the corresponding flip distance will be \NP-hard to compute.
On the other hand, we use very particular graphs as gadgets, and we do not know the complexity of the corresponding problem for planar $\alpha$-orientations.
We think the following is an interesting special case:

\begin{question}
Let $X, Y$ be perfect matchings of a planar bipartite $3$-connected graph $G$.
What is the complexity of determining the distance of $X$ and $Y$ with respect to alternating cycles that are either a face or the symmetric difference of two incident faces?
\end{question}
The feeling that this problem might be tractable is supported by the following observation. It is not difficult to show that every height two poset with bipartite \emph{planar} Hasse diagram has dimension at most two.
It then follows from~\cite{SS87} that the restriction of the Jump Number Problem to such posets is solvable in polynomial time, and thus, the hardness reduction presented in the previous section fails.

\subsection*{Acknowledgements}
This work was initiated during the workshop ``Order \& Geometry'' 2018 in Ci\k{a}\.{z}e\'{n} Palace. We thank the organizers and participants of this workshop for the stimulating atmosphere.

\bibliographystyle{alpha}
\bibliography{paper}

\end{document}